\documentclass[12pt]{amsart}
\usepackage{amssymb,amsmath,amsthm,amssymb,amsfonts,amscd}
\usepackage{graphicx}
\usepackage{bm}
\usepackage{color, enumerate}
\usepackage[all]{xy}
\usepackage[english]{babel}
\usepackage{dsfont} 
\usepackage{verbatim}

\usepackage{etoolbox}
\patchcmd{\thmhead}{(#3)}{#3}{}{}

\newtheorem{theorem}{Theorem}[section]
\newtheorem{lemma}[theorem]{Lemma}

\newtheorem{corollary}[theorem]{Corollary}
\newtheorem{proposition}[theorem]{Proposition}
\newtheorem{remark}[theorem]{Remark}

\DeclareMathAlphabet{\mathpzc}{OT1}{pzc}{m}{it}

\newcommand{\R}{\mathbb{R}}
\newcommand{\N}{\mathbb{N}}

\def\supp{{\rm supp}}

\def\N{{\mathbb{N}}}

\def\R{{\mathbb{R}}}

\def\0{{\rm \bf{0}}}

\def\B{{\mathcal{B}}}

\def\q{{\vec{q}}}
\def\qm{\bm{q}}

\def\B2star{\overline{B}_X^{w(X^{\ast\ast},X^{\ast})}}

\title{Multilinear Marcinkiewicz-Zygmund inequalities}
\author{Daniel Carando %
\and Martin Mazzitelli %
\and Sheldy Ombrosi}

\thanks{This project was supported in part by CONICET PIP 11220130100329CO, ANPCyT PICT 2015-2299 and UBACyT 20020130100474. The second author has a postdoctoral position from CONICET}

\address{Departamento de Matem\'{a}tica - Pab I,
Facultad de Cs. Exactas y Naturales, Universidad de Buenos Aires,
(1428) Buenos Aires, Argentina and IMAS-CONICET}
\email{dcarando@dm.uba.ar}

\address{Instituto Balseiro,
Universidad Nacional de Cuyo - C.N.E.A. and Departamento de Matem\'atica, Centro Regional Universitario Bariloche, Universidad Nacional del Comahue, (8400) San Carlos de Bariloche, Argentina}
\email{martin.mazzitelli@crub.uncoma.edu.ar}

\address{Departamento de Matem\'{a}tica,
Universidad Nacional del Sur,
(8000) Bah\'ia Blanca, Argentina and INMABB-CONICET}
\email{sombrosi@uns.edu.ar}

\keywords{Vector-valued inequalities, multilinear operators, Calder\'on-Zygmund operators}
\subjclass[2010] {Primary: 47H60, 47A63. Secondary: 42B20}

\date{}

\begin{document}
\baselineskip=.65cm

\begin{abstract}
We extend to the multilinear setting classical inequalities of Marcinkiewicz and Zygmund on $\ell^r$-valued extensions of linear operators. We show that for certain $1 \leq p, q_1, \dots, q_m, r \leq \infty$, there is a constant $C\geq 0$ such that for every bounded multilinear operator $T\colon L^{q_1}(\mu_1) \times \cdots \times L^{q_m}(\mu_m) \to L^p(\nu)$ and  functions $\{f_{k_1}^1\}_{k_1=1}^{n_1} \subset L^{q_1}(\mu_1), \dots, \{f_{k_m}^m\}_{k_m=1}^{n_m} \subset L^{q_m}(\mu_m)$, the following inequality holds
\begin{equation}\label{MZ ineq abstract}
\left\Vert \left(\sum_{k_1, \dots, k_m} |T(f_{k_1}^1, \dots, f_{k_m}^m)|^r\right)^{1/r} \right\Vert_{L^p(\nu)} \leq C \|T\| \prod_{i=1}^m \left\| \left(\sum_{k_i=1}^{n_i} |f_{k_i}^i|^r\right)^{1/r} \right\|_{L^{q_i}(\mu_i)}.
\end{equation}
In some cases we also calculate the best constant $C\geq 0$ satisfying the previous inequality.
We apply these results to obtain weighted vector-valued inequalities for multilinear Calder\'on-Zygmund operators. 
\end{abstract}

\maketitle

\section{Introduction and main results}

The study of vector-valued inequalities for linear operators has its origins in the thirties, with works of Bochner, Marcinkiewickz, Paley and Zygmund among others.
In this context we find the so-called Marcinkiewicz-Zygmund inequalities for linear operators, regarding the $\ell^r$-valued extensions of linear operators between real $L^p$-spaces. That is, given $1\leq p,q,r \leq \infty$, the triple $(p,q,r)$ is said to satisfy a Marcinkiewicz-Zygmund inequality if there is a constant $C$ such that for each bounded operator $T\colon L^q(\mu) \to L^p(\nu)$,
($\mu$ and $\nu$ arbitrary $\sigma$-finite measures, $L^q(\mu)$ and $L^p(\nu)$ real spaces),
each $n \in \N$ and  functions $f_1,\dots,f_n \in L^q(\mu)$,
\begin{equation}\label{MZ property}
\left(\int \left(\sum_{k=1}^n |T(f_k)(\omega)|^r\right)^{p/r} d\nu(\omega)\right)^{1/p} \leq C \|T\| \left(\int \left(\sum_{k=1}^n |f_k(\omega)|^r\right)^{q/r} d\mu(\omega)\right)^{1/q}.
\end{equation}
The infimum of all the constants $C\geq 1$ satisfying \eqref{MZ property} is denoted by $k_{q,p}(r)$ (setting $k_{q,p}(r)=\infty$ if there is not such constant).  Fixed $n \in \N$, let $k_{q,p}^{(n)}(r) \in [1, \infty)$ be the infimum of all the constants $C\geq 0$ satisfying \eqref{MZ property} for each $T$ but for only $n$ functions $f_k$.
Note that given $T\colon L^q(\mu) \to L^p(\nu)$ we can consider the natural $\ell^r_n$-valued extension $T^{\ell^r_n} \colon L^q(\mu, \ell^r_n) \to L^p(\nu, \ell^r_n)$ defined by
\begin{equation}\label{l_r valued linear operator}
T^{\ell^r_n}((f^1, \dots, f^n)) := (Tf^1, \dots, Tf^n).
\end{equation}
It is clear that $k_{q,p}^{(n)}(r)=\sup \|T^{\ell^r_n}\|$, where the supremum is taken over all the operators $T \colon L^q(\mu) \to L^p(\nu)$ with norm $\|T\|\leq 1$. Analogously, the validity of \eqref{MZ property} is equivalent to saying that each $T \colon L^q(\mu) \to L^p(\nu)$ has a natural $\ell^r$-valued extension and, in that case,
$
k_{q,p}(r) = \sup \|T^{\ell^r} \colon L^q(\mu, \ell^r) \to L^p(\nu, \ell^r) \|,
$
where the supremum is taken over all the operators $T \colon L^q(\mu) \to L^p(\nu)$ with norm $\|T\| \leq 1$. It is known (see \cite[29.12]{DefFlo}) that this supremum does not change if it is taken over two fixed measures $\mu$ and $\nu$ such that $L^q(\mu)$ and $L^p(\nu)$ are infinite-dimensional.
It is worth mentioning that the problem of determining the constants $k_{q,p}^{(n)}(r)$ is a generalization of the complexification problem, that is, the computation of the so-called \emph{complexification constants} ($k_{q,p}^{(2)}(2)$ in our terminology) of operators in real $L_p$-spaces, which relate the norm of an operator and its complexification.

In \cite{MarZyg}, Marcinkiewicz and Zygmund proved that $k_{q,p}(2)<\infty$ for $1 \leq p,q < \infty$ and $k_{q,p}(r)<\infty$ whenever $1 \leq \max(p,q)<r<2$. In these cases, they also obtained estimates for the constants $k_{q,p}(r)$ in terms of the $q$-th moment of $r$-stable L\'evy measures. In the particular case $q=p$, they obtained $k_{p,p}(2)=1$. Herz extended this last equality in \cite{Her}, showing that $k_{p,p}(r)=1$ for $1<p<\infty$ and $\min(p,2)\leq r \leq \max(p,2)$, and Grothendieck established the important case $K_{G,\R}:=k_{\infty,1}(2)<\infty$ ($K_{G,\R}$ stands for the so-called Grothendieck constant in the real case).
A systematic study of the constants $k_{q,p}(r)$ and the precise asymptotic growth of $k_{q,p}^{(n)}(r)$ is addressed in \cite{DefJun, GasMal}. In the following theorem, we state some properties of the constants $k_{q,p}(r)$ (including monotonicity and duality) that can be found in the mentioned papers.

\begin{theorem}[\cite{DefJun, GasMal}] \label{propiedades de los kqp}
\begin{enumerate}
\item[\rm (i)]
$
k_{q,p}(r) = \lim_{n \to \infty} k_{q,p}^{(n)}(r).
$
\item[\rm (ii)] $k_{q_1,p_1}^{(n)}(r) \leq k_{q_2,p_2}^{(n)}(r)$ whenever $q_1 \leq q_2$ and $p_2 \leq p_1$.
\item[\rm (iii)] As a function of $r$, $k_{q,p}(r)$ is decreasing on $[1,2]$ and increasing on $[2, \infty]$.
\item[\rm (iv)] $k_{q,p}(r) = k_{p', q'}(r')$.
\end{enumerate}
\end{theorem}

In \cite{DefJun}, the set of all the triples $1 \leq p,q,r \leq \infty$ satisfying $k_{q,p}(r) < \infty$ is determined and also the exact value of this constant is obtained in almost all the cases. We state as a theorem these results, that can be found in Sections~4, 5 and 6 of the aforementioned article.

\begin{theorem}[\cite{DefJun}]\label{constantes kqp para todos los triples}
\begin{enumerate}
\item[\rm (i)] If $q=1$ or $p=\infty$, then $k_{q,p}(r)=1$.
\item[\rm (ii)] Let $1 < q \leq p < \infty$. Then $k_{q,p}(r) < \infty$ if and only if $\min\{q,2\} \leq r \leq \max\{p,2\}$. Moreover, $k_{q,p}(r)=1$ in that case.
\item[\rm (iii)] Let $1 \leq p \leq q \leq \infty$ with $q\neq 1$ and $p\neq \infty$ (these cases are considered in (i)). Then $k_{q,p}(r) < \infty$ if and only if one of the following cases holds:
\begin{itemize}
\item $p=q=r$, in which case $k_{q,p}(r)=1$;
\item $1 \leq p \leq q \leq 2$ and $q < r \leq 2$, in which case $k_{q,p}(r) = \frac{c_{r,q}}{c_{r,p}}$;
\item $2 \leq p \leq q \leq \infty$ and $2 \leq r <p$, in which case $k_{q,p}(r) = \frac{c_{r',p'}}{c_{r',q'}}$;
\item $1 \leq p \leq 2 \leq q \leq \infty$ and $r=2$; if, moreover, $p=2$ or $q=2$ then $k_{q,p}(2) = \frac{c_{2,q}}{c_{2,p}}$.
\end{itemize}
The constant $c_{r,q}$ denotes the $q$-th moment of $r$-stable L\'evy measure.
The only case in which $k_{q,p}(r)$ is not determined is when $1 \leq p < r=2 < q \leq \infty$.
\end{enumerate}
\end{theorem}

We are interested in the study of Marcinkiewickz-Zygmund inequalities in the context of multilinear operators. In what follows, given $\{f_k\}_k \subset L^p(\nu)$, we denote
$$
\left\| \left(\sum_k |f_k|^r \right)^{1/r} \right\|_{L^p(\nu)}:=\left( \int \left(\sum_k |f_k(\omega)|^r \right)^{p/r} d\nu(\omega)  \right)^{1/p}.
$$
Let $m \in \N$, $1\leq q_1, \dots, q_m, p, r \leq \infty$ and consider $\q=(q_1, \dots, q_m)$.
We say that the triple $(p;\q;r)$ satisfies the Marcinkiewicz-Zygmund inequality if  there is a constant $C$ such that for all bounded multilinear operators $T\colon L^{q_1}(\mu_1) \times \cdots \times L^{q_m}(\mu_m) \to L^p(\nu)$
and  functions $\{f_{k_1}^1\}_{k_1=1}^{n_1} \subset L^{q_1}(\mu_1), \dots, \{f_{k_m}^m\}_{k_m=1}^{n_m} \subset L^{q_m}(\mu_m)$, the following inequality holds (with the usual modification when $r=\infty$)
\begin{equation}\label{multilinear MZ property}
\left\| \left(\sum_{k_1, \dots, k_m} |T(f_{k_1}^1, \dots, f_{k_m}^m)|^r\right)^{1/r}\right\|_{L^p(\nu)} \leq C \|T\| \prod_{i=1}^m \left\| \left(\sum_{k_i=1}^{n_i} |f_{k_i}^i|^r\right)^{1/r} \right\|_{L^{q_i}(\mu_i)}.
\end{equation}
As in the linear case, we denote
by $k_{\q,p}(r)$ the infimum of all the $C\geq 1$ satisfying \eqref{multilinear MZ property} and we put $k_{\q,p}(r)=\infty$ when there is no such constant.
We denote by $k_{q,p}^{(n)}(r) \in [1, \infty)$ the infimum of all the constants $C\geq 0$ satisfying \eqref{multilinear MZ property} only for $n_1=\cdots = n_m=n$. It is easy to see that $\lim_{n \to \infty} k_{\q,p}^{(n)}(r) = k_{\q,p}(r)$.
As observed in \eqref{l_r valued linear operator} for linear operators, the validity of \eqref{multilinear MZ property} is equivalent to saying that each $T\colon L^{q_1}(\mu_1) \times \cdots \times L^{q_m}(\mu_m) \to L^p(\nu)$ has the $\ell^r$-valued extension $T^{\ell^r}\colon L^{q_1}(\mu_1, \ell^r) \times \cdots \times L^{q_m}(\mu_m, \ell^r) \to L^p(\nu, \ell^r(\N \times \cdots \times\N))$ defined by
\begin{equation}\label{l_r valued multilinear operator}
T^{\ell^r}\left((f_{k_1}^1)_{k_1}, \dots, (f_{k_m}^m)_{k_m} \right) = \left(T(f_{k_1}^1, \dots, f_{k_m}^m)\right)_{k_1, \dots, k_m}.
\end{equation}
In that case $k_{\q,p}(r) = \sup \|T^{\ell^r}\|$, where the supremum is taken over all the multilinear operators $T \colon L^{q_1}(\mu_1) \times \cdots \times L^{q_m}(\mu_m) \to L^p(\nu)$  (and over all measures $\mu_i$ and $\nu$) with norm $\|T\|\leq 1$.

As we will point out in Section~\ref{comparison}, it is worth mentioning that, when dealing with inequalities of the form
$$
\left\| \left( \sum_{k_1, \dots, k_m} |T(f^1_{k_1}, \dots, f^m_{k_m})|^s \right)^{\frac{1}{s}}  \right\|_{L^p(\nu)} \leq C \prod_{i=1}^m \left\| \left( \sum_{k_i} |f^i_{k_i}|^{r_i} \right)^{\frac{1}{r_i}} \right\|_{L^{q_i}(\mu_i)},
$$
the relation between the powers $s, r_1, \dots, r_m$ is optimal when $s=r_1=\cdots = r_m$, as in \eqref{multilinear MZ property}. This establishes a difference with the inequalities of the form
$$
\left\| \left( \sum_k |T(f^1_k, \dots, f^m_k)|^r \right)^{\frac{1}{r}} \right\|_{L^p(\nu)} \leq C \prod_{i=1}^m \left\| \left( \sum_k |f^i_k|^{r_i} \right)^{\frac{1}{r_i}} \right\|_{L^{q_i}(\mu_i)}
$$
where the sum runs over only one index $k$ and the optimal relation between the powers is given by $\frac{1}{r}=\frac{1}{r_1}+\cdots + \frac{1}{r_m}$.

In \cite{GraMar}, in the context of vector-valued inequalities for multilinear Calder\'on-Zygmund operators, Grafakos and Martell addressed the multilinear  Marcinkiewicz-Zygmund inequality in the particular case $r=2$, showing  that $k_{\q,p}(2)<\infty$ for every $1\leq q_1, \dots, q_m, p<\infty$ (moreover, for $0<q_1, \dots, q_m, p<\infty$). They also proved the analogous inequality for multilinear operators from $ L^{q_1}(\mu_1) \times \cdots \times L^{q_m}(\mu_m)$ into $ L^{p,\infty}(\nu)$. Independently, Bombal, P\'erez-Garc\'ia and Villanueva proved in \cite[Thm. 4.2]{BomPerVil} a  multilinear  Marcinkiewicz-Zygmund inequality in the case $r=2$ in the more general context of multilinear operators defined on Banach lattices.
Our goal is to determine conditions on $p, \q,r$ such that the triple $(p,\q,r)$ satisfies the Marcinkiewicz-Zygmund inequality \eqref{multilinear MZ property} and, if it is possible, to calculate the exact value of the constant $k_{\q,p}(r)$. In this sense, our main result is the following.
\begin{theorem}\label{main theorem}
Let $1 \leq q_1, \dots, q_m \leq \infty$, $1 \leq p < \infty$ and $\qm=\max\{q_1, \dots, q_m\}$.
\begin{enumerate}
\item[\rm (i)] Suppose $\qm \leq p$.
\begin{enumerate}
\item[\rm (ia)] If $1 \leq \qm \leq p \leq 2$, then $k_{\q,p}(r)< \infty$ iff $\qm \leq r \leq 2$ or $\qm=1$ and $1 \leq r\leq \infty$.
\item[\rm (ib)] If $2 \leq \qm \leq p$, then $k_{\q,p}(r)< \infty$ iff $2 \leq r \leq p$.
\item[\rm (ic)] If $1 \leq \qm \leq 2 \leq p$, then $k_{\q,p}(r)< \infty$ iff $\qm \leq r \leq p$ or $\qm=1$ and $1 \leq r\leq \infty$.
\end{enumerate}

\item[\rm (ii)] Suppose $p < \qm$.
\begin{enumerate}
\item[\rm (iia)] If $1 \leq p < \qm \leq 2$, then $k_{\q,p}(r)< \infty$ iff $\qm < r \leq 2$ or $\qm=r=2$.
\item[\rm (iib)] If $2 < p < \qm$, then $k_{\q,p}(r)< \infty$ implies $2 \leq r < p$. If $r=2$ then $k_{\q,p}(r)<\infty$.
\item[\rm (iic)] If $1 \leq p \leq 2 \leq \qm$, then $k_{\q,p}(r)< \infty$ iff $r=2$.
\end{enumerate}
\end{enumerate}
If {\rm (ia), (ic)}, {\rm(iia)} holds or if $\qm \leq r \leq p$, then $k_{\q,p}(r)= k_{q_1,p}(r)\cdots k_{q_m,p}(r)$.
\end{theorem}
Note that the only case in which we do not get an equivalence is in (iib), where we obtain a necessary condition for $k_{\q,p}(r)< \infty$ (and the, already known, sufficient condition $r=2$). In Section~\ref{applications} we address the case $0< p < 1$ and the case of weak type estimates, which will be relevant in obtaining vector-valued estimates for multilinear singular integrals.

We also obtain some properties of the constants $k_{\q,p}(r)$ that, besides being important in the proof of the previous theorem, are interesting on their own. First, we study the relation between  $k_{\q,p}(r)$ and the \textit{linear} constants $k_{q_1,p}(r), \dots, k_{q_m,p}(r)$. We get the following result which, in particular, shows that  if $k_{q_i,p}(r)=\infty$ for some $1\leq i\leq m$, then $k_{\q,p}(r)=\infty$.
\begin{proposition} \label{relacion lineal-multilineal}
Let $1 \leq p,q_1, \dots, q_m, r \leq \infty$. Then, for each $n \in \N$ we have
\begin{equation}\label{desigualdad lineal-multilineal}
k_{q_1;p}^{(n)}(r) \cdots k_{q_m;p}^{(n)}(r) \leq k_{\q;p}^{(n)}(r).
\end{equation}
Consequently, $k_{q_1;p}(r) \cdots k_{q_m;p}(r) \leq k_{\q;p}(r)$. Moreover, if $p=r$ then equality holds.
\end{proposition}
We also prove the following monotonicity properties of $k_{\q, p}(r)$ as a function of $p$ and $r$, partially extending to multilinear setting the properties (ii) and (iii) of Theorem~\ref{propiedades de los kqp}.
\begin{proposition}\label{monotonia en p y r}
Let $1\leq q_1, \dots, q_m,p \leq \infty$.
\begin{enumerate}
\item[\rm (i)] If  $1\leq p_2<p_1 \leq \infty$, then $k_{\q,p_1}^{(n)}(r) \leq k_{\q,p_2}^{(n)}(r)$ for all $1\leq r \leq \infty$. Consequently, $k_{\q,p_1}(r) \leq k_{\q,p_2}(r)$.
\item[\rm (ii)] If $1\leq r<s \leq 2$, then $k_{\q,p}(s) \leq k_{\q,p}(r)$.
\end{enumerate}
\end{proposition}

Although Theorem~\ref{main theorem} seems to suggest that $k_{\q,p}(r) < \infty$ if and only if $k_{q_i,p}(r)<\infty$ for all $1 \leq i \leq m$,
we establish an important difference between the Marcinkiewicz-Zygmund inequalities in the linear and multilinear cases when, for instance, we put $p=\infty$. First, as a simple consequence of the optimality in the well-known Littlewood's $4/3$ inequality, we see this different behavior between the linear and bilinear cases when we put $q_1=q_2=p=\infty$.

\begin{remark}\label{rem-lit}
Denote $\vec{\infty}=(\infty,\infty)$. Then, although $k_{\infty,\infty}(r)=1$ for all $1\leq r\leq \infty$ (this was stated in Theorem~\ref{constantes kqp para todos los triples}), we have $k_{\vec{\infty},\infty}(r)=\infty$ whenever $1\leq r < 4/3$.
Indeed, suppose that $k_{\vec{\infty},\infty}(r) < \infty$. 
Then, for any $n \in \N$ and $T\colon \ell^\infty_n \times \ell^\infty_n \to \R$ we  have
$$
 \left(\sum_{k_1, k_2=1}^n |T(e_{k_1}, e_{k_2})|^r\right)^{1/r} \leq k_{\vec{\infty},\infty}(r) \|T\| \prod_{i=1}^2 \left\| \left(\sum_{k_i=1}^{n} |e_{k_i}|^r\right)^{1/r} \right\|_{\ell^\infty_n} = k_{\vec{\infty},\infty}(r) \|T\|.
$$
Thus, we  obtain
$
 \left(\sum_{k_1, k_2=1}^\infty |T(e_{k_1}, e_{k_2})|^r\right)^{1/r} \leq k_{\vec{\infty},\infty}(r) \|T\|
$
for every bilinear form $T\colon c_0\times c_0 \to \R$. The optimality in Littlewood's $4/3$ inequality implies that $r\ge4/3$.
\end{remark}

The next proposition generalizes the previous remark showing that, in contrast to the linear case where $k_{q, \infty}(r) =1$ for all $1 \leq q, r \leq \infty$, in the $m$-linear case $k_{\q, \infty}(r)=\infty$ for appropriate choices of $m \in \N$ and $1 \leq q_1, \dots, q_m,r \leq \infty$. In particular, we see that $k_{q_i,p}(r)<\infty$ for all $1 \leq i\leq m$ does not imply, in general, $k_{\q,p}(r)<\infty$. It should be noted that, just as in Remark~\ref{rem-lit} the optimality of Littlewood's $4/3$ inequality was used, Proposition~\ref{proposicion caso infinito} is related to the optimality of some of its  multilinear versions \cite{DefSev}, such as the Bohnenblust-Hille inequality \cite{BoHi} (see also \cite{Dav,Kai}). In fact, our proof follows the spirit of the new approaches to these inequalities by means of probabilistic method \cite{Que} (see Lemma~\ref{KSZ ineq}).

\begin{proposition}\label{proposicion caso infinito}
Let $1 \leq q_1, \dots, q_m \leq \infty$ and $\qm = \max\{q_1, \dots, q_m\}$. Then,  for
\begin{equation}\label{eq-casoinfinito}1 \le r  <  m \left( \frac{1}{\max(\qm',2)} + \frac{1}{\min(q_1,2)}+\cdots +\frac{1}{\min(q_m,2)}\right)^{-1}\end{equation}
we have $k_{\q, \infty}(r) = \infty$, while  $k_{q_i, \infty}(r)=1$, $1 \leq i \leq m$.
\end{proposition}

\begin{remark}\rm It is easy to check that the expression in the right hand side of \eqref{eq-casoinfinito} is smaller than $\min\{\qm, 2\}$. On the other hand,  since $k_{\q,\infty}(\min\{\qm, 2\})<\infty$ and $k_{\q,\infty}(\infty)=1$ (see Theorem~\ref{main theorem} (ib) and Remark~\ref{remark after monotonicity}), by Corollary~\ref{interpolation property of r} below  we have $k_{\q,\infty}(r)< \infty$ for $\min\{\qm, 2\} \leq r \leq \infty$. This brings up the question on the behaviour of $k_{\q,\infty}(r)$ for the remaining range of values of $r$.  Although the multilinear versions of Littlewood inequalities are understood for the full range of exponents, we were not be able to fill the gap for $r$ in our results.

\end{remark}

As the nature of Marcinkiewicz-Zygmund inequalities is related to the study of $\ell^r$-valued extensions of bounded multilinear operators on $L^p$-spaces, a brief review of some of the inequalities involving Theorem~\ref{main theorem} allows us to obtain vector-valued estimates for multilinear singular integrals. For instance, we prove that if
$1 < q_1, \dots, q_m <r < 2$ and $p>0$ are such that
$
\frac{1}{p}=\frac{1}{q_1}+\cdots+\frac{1}{q_m}
$
and $\vec w=(w_1,\dots,w_m)$ satisfies the multilinear $A_{\q}$ condition defined in~ \cite{LOPTT}, then there exists a constant $C>0$ such that for every Calder\'on-Zygmund multilinear operator $T$ we have
$$
\left\|  \left( \sum_{k_1, \dots, k_m} |T(f^1_{k_1}, \dots, f^m_{k_m})|^r \right)^{\frac{1}{r}} \right\|_{L^p(\nu_{\vec w})} \leq C \|T\| \prod_{i=1}^m \left\| \left( \sum_{k_i} |f^i_{k_i}|^{r} \right)^{\frac{1}{r}} \right\|_{L^{q_i}(w_i)},
$$
where $\nu_{\vec w}=\prod_{i=1}^m w_i^{p/q_i}$. We address the expected weak type estimates when some of the exponents $q_i$ are equal to one.
We compare our results with some known vector-valued estimates obtained in \cite{BenMus, BenMus2, CruMarPer, GraMar}.

Finally, we present the multilinear version of \cite[Chapter V, Thm. 1.12]{GarRub}, which states that if $T\colon L^q(\mu) \to L^p(\nu)$ is a positive linear operator (that is, $f\geq 0$ implies $T(f)\geq 0$), then the Marcinkiewicz-Zygmund inequality \eqref{MZ property} holds for all $1 \leq r \leq \infty$, with $C=1$.

\begin{proposition}\label{MZ positive}
Let $0< p, q_1, \dots, q_m \leq \infty$, $1 \leq r \leq \infty$ and $T\colon L^{q_1}(\mu_1)\times \cdots L^{q_m}(\mu_m) \to L^p(\nu)$ be a positive multilinear operator. Then
\begin{equation}\label{ineq MZ positive}
\left\|\left(\sum_{k_1, \dots, k_m} |T(f_{k_1}^1, \dots, f_{k_m}^m)|^r\right)^{1/r}\right\|_{L^p(\nu)} \leq  \|T\| \prod_{i=1}^m \left\| \left(\sum_{k_i=1}^{n_i} |f_{k_i}^i|^r\right)^{1/r} \right\|_{L^{q_i}(\mu_i)}
\end{equation}
for any choice of functions $\{f_{k_i}^i\}_{k_i=1}^{n_i} \subset L^{q_i}(\mu_i)$, $1\leq i \leq m$.
\end{proposition}
As an immediate consequence, we obtain the estimate
$$
\left\|\left(\sum_{k_1, k_2} |f_{k_1}^1 * f_{k_2}^2|^r\right)^{1/r}\right\|_{L^p(\R)} \leq  \left\| \left(\sum_{k_1=1}^{n_1} |f_{k_1}^1|^r\right)^{1/r} \right\|_{L^{q_1}(\R)} \left\| \left(\sum_{k_2=1}^{n_2} |f_{k_2}^2|^r\right)^{1/r} \right\|_{L^{q_2}(\R)}
$$
for the bilinear convolution operator, whenever $\frac{1}{q_1} + \frac{1}{q_2} = \frac{1}{p} + 1$.

\subsection*{Structure of the article} The article is organized as follows. In Section~\ref{propiedades de las k} we prove the mentioned properties of the constants $k_{\q,p}(r)$. The inequality $k_{q_1;p}(r) \cdots k_{q_m;p}(r) \leq k_{\q;p}(r)$ is almost immediate, while the equality in the case $p=r$ follows by induction. For the monotonicity of $k_{\q,p}(r)$ as a function of $p$ we use a duality argument. The monotonicity in $r$ makes use of $r$-stable L\'evy measures and a generalization of some arguments in \cite{GasMal}. We include in this section a kind of interpolation property of $k_{\q,p}(r)$ as a function of $r$ (see Corollary~\ref{interpolation property of r}) and we also prove that $k_{\q,p}(r)=1$ when $\q= (1, \dots, 1)$.
Our main results are given in Section~\ref{section main results}, where we determine conditions on $\q,p$ and $r$ so that $k_{\q,p}(r) < \infty$ and we obtain (in many cases) the exact value of these constants.
In Theorem~\ref{caracterizacion caso menor 2} we focus on the case $1\leq q_1, \dots, q_m \leq 2$, $1 \leq p < \infty$. The proof of this theorem combines some classical arguments with properties from Section~\ref{propiedades de las k} and the values of the constants $k_{q_i,p}(r)$ obtained in \cite{DefJun}. In Section~\ref{section main theorem} we include the proof of Theorem~\ref{main theorem}, which at that point is just a combination of the results previously obtained.
In Section~\ref{caso p infinito} we treat the case $p=\infty$, showing (see Proposition~\ref{proposicion caso infinito}) that the behavior of $k_{\q,\infty}(r)$ is very different from that of $k_{q_i, \infty}(r)$. The proof relies on a multilinear version of a Kahane-Salem-Zygmund inequality.
In Section~\ref{applications} we discuss some applications of the Marcinkiewicz-Zygmund inequalities. We deal with weighted vector-valued inequalities for multilinear Calder\'on-Zygmund operators
in Section~\ref{vector-valued CZ operators} and, in Section~\ref{section MZ positive}, we prove Proposition~\ref{MZ positive}, which gives vector-valued inequalities for positive multilinear operators (such as the convolution).

\subsection*{Notation} All the Banach spaces considered are real and all the measure spaces $(\Omega, \mu)$ are $\sigma$-finite. As usual, given a measure space $(\Omega, \mu)$ the space of measurable functions $f\colon \Omega \to \R$ such that $\int_{\Omega}|f(\omega)|^p\, d\mu(\omega) < \infty$ is denoted by $L^p(\mu)=L^p(\Omega; d\mu)$ (we omit the set $\Omega$ which will be clear by context).
Given a Banach space $X$, we denote by $L^p(\mu, X)$ the vector space of $X$-valued $\mu$-measurable functions $f \colon \Omega \to X$ such that $\|f(\cdot)\|_{X}^p$ is $\mu$-integrable. When $1 \leq p \leq \infty$, this is a Banach space with the natural norm
$
\|f\|_{L^p(\mu, X)} = \left(\int \|f(\omega)\|_{X}^p d\mu(\omega)\right)^{1/p}
$
(the case $p=\infty$ is defined analogously).
Following the standard notation, we denote by $\ell^p$ the space of sequences $(a_k)_{k\in \N}$ in $\R$ such that $\|(a_k)_k\|_{\ell^p} = \left( \sum_k  |a_k|^p\right)^{1/p} < \infty$ and $\ell^p_n = (\R^n, \|\cdot\|_{\ell^p})$, with the usual modification when $p=\infty$.

\section{Basic properties of $k_{\q,p}(r)$}\label{propiedades de las k}

\subsection{Relation between the linear and multilinear constants}

Before proving Proposition~\ref{relacion lineal-multilineal} we state  the following easy remark, whose proof is omitted. Recall that if $(\Omega_1, \nu_1),\dots, (\Omega_m, \nu_m)$ are measure spaces, then $\nu_1 \otimes \cdots \otimes \nu_m$ denotes the product measure on $\Omega_1 \times \cdots \times \Omega_m$.
\begin{remark} \label{bilineal}
Let $T_i \colon L^{q_i}(\mu_i) \to L^p(\nu_i)$ be bounded  operators ($i=1,\dots, m$) and let $T \colon L^{q_1}(\mu_1) \times \cdots \times L^{q_m}(\mu_m) \to L^p(\nu_1 \otimes \cdots \otimes \nu_m)$ be the $m$-linear operator defined as
$
T(f^1, \dots, f^m)= T_1(f^1) \otimes \cdots \otimes T_m(f^m),
$
, i.e.,
$$T(f^1, \dots, f^m)(\omega_1, \dots, \omega_m)=T_1(f^1)(\omega_1) \cdots T_m(f^m)(\omega_m).$$ Then $\|T\|=\|T_1\| \cdots \|T_m\|$ and
\begin{equation}\label{relacion normas p}
\left\| \left(\sum_{k_1, \dots, k_m=1}^n |T(f_{k_1}^1, \dots, f_{k_m}^m)|^r\right)^{1/r} \right\|_{L^p(\nu_1 \otimes \cdots \otimes \nu_m)} = \prod_{i=1}^m \left\|\left(\sum_{k_i=1}^n |T_i(f_{k_i}^i)|^r\right)^{1/r}\right\|_{L^p(\nu_i)}
\end{equation}
for each $n \in \N$ and $\{f_{k_i}^i\}_{k_i=1}^n \subset L^{q_i}(\mu_i)$.
\end{remark}

\begin{proof}[Proof of Proposition~\ref{relacion lineal-multilineal}]
It is clear that $k_{q_1;p}^{(n)}(r) \cdots k_{q_m;p}^{(n)}(r)$ is the infimum of all the constants $C\geq 0$ satisfying
\begin{equation*}
\prod_{i=1}^m \left\|\left(\sum_{k_i=1}^n |T_i(f_{k_i}^i)|^r\right)^{1/r}\right\|_{L^p(\nu_i)} \leq C \prod_{i=1}^m \|T_i\| \left\| \left(\sum_{k_i=1}^n |f_{k_i}^i|^r\right)^{1/r} \right\|_{L^{q_i}(\mu_i)}
\end{equation*}
for every $T_i \colon L^{q_i}(\mu_i) \to L^p(\nu_i)$
and all functions $\{f_{k_i}^i\}_{k_i=1}^n \subset L^{q_i}(\mu_i)$ ($1 \leq i \leq  m$). By the previous remark, we see that $k_{q_1;p}^{(n)}(r) \cdots k_{q_m;p}^{(n)}(r)$ is the infimum of all the constants $C\geq 0$ satisfying
\begin{equation}\label{MZ bilineal}
\left\| \left(\sum_{k_1, \dots, k_m=1}^n |T(f_{k_1}^1, \dots, f_{k_m}^m)|^r\right)^{1/r} \right\|_{L^p(\nu_1 \otimes \cdots \otimes \nu_m)} \leq C \|T\| \prod_{i=1}^m \left\| \left(\sum_{k_i=1}^n |f_{k_i}^i|^r\right)^{1/r} \right\|_{L^{q_i}(\mu_i)}
\end{equation}
for every $m$-linear operator $T\colon L^{q_1}(\mu_1) \times \cdots \times L^{q_m}(\mu_m) \to L^p(\nu_1 \otimes \cdots \otimes \nu_m)$ of the form $T(f^1, \cdots, f^m)=T_1(f^1)\otimes \cdots \otimes T_m(f^m)$ and
functions $\{f_{k_i}^i\}_{k_i=1}^n \subset L_{q_i}(\mu_i)$. This shows that the constant $k_{\q,p}^{(n)}(r)$, which is the infimum of all the constants satisfying \eqref{MZ bilineal} for every $m$-linear operator from $L^{q_1}(\mu_1) \times \cdots \times L^{q_m}(\mu_m)$ to $L^p(\nu)$,
is greater than or equal to $k_{q_1;p}^{(n)}(r) \cdots k_{q_m;p}^{(n)}(r)$.

To prove equality when $p=r$, we reason by induction on $m$ (see, for instance, \cite[Thm. 3.1]{BomPerVil} for a similar argument).
For $m=1$ the result is trivial, then let $m \geq  2$ and suppose that the result holds for $m-1$. We write the proof for $1 \leq r<\infty$, the case $r=\infty$ being analogous.
Fix $T\colon L^{q_1}(\mu_1) \times \cdots \times L^{q_m}(\mu_m) \to L^r(\nu)$, $n \in \N$ and $\{f_{k_i}^i\}_{k_i=1}^{n} \subset L^{q_i}(\mu_i)$. Our goal is to prove
\begin{equation*}
\left\| \left(\sum_{k_1, \dots, k_m} |T(f_{k_1}^1, \dots, f_{k_m}^m)|^r\right)^{1/r} \right\|_{L^r(\nu)} \leq k_{q_1,r}^{(n)}(r) \cdots k_{q_m,r}^{(n)}(r) \|T\| \prod_{i=1}^m \left\| \left(\sum_{k_i=1}^{n} |f_{k_i}^i|^r\right)^{1/r} \right\|_{L^{q_i}(\mu_i)},
\end{equation*}
which would give the remaining inequality $ k_{\q,r}^{(n)}(r) \leq k_{q_1,r}^{(n)}(r) \cdots k_{q_m,r}^{(n)}(r)$. For each $2\leq i \leq m$ let $\nu_i$ be the counting measure on $\N$ and
define $S\colon L^{q_1}(\mu_1) \to L^r(\nu \otimes \nu_2 \otimes \cdots \otimes \nu_m)$ by
$$
S(f^1)(\omega, k_2, \dots, k_m)=
 \left\{
\begin{array} {c l}
T(f^1, f_{k_2}^2, \dots, f_{k_m}^m)(\omega) & \text{if $k_2, \dots, k_m \leq n$,}\\
0  & \text{otherwise.}
\end{array}
\right .
$$
It is easy to check that
\begin{equation}\label{relacion S y T}
\left\| \left( \sum_{k_1=1}^n |S(f_{k_1}^1)|^r \right)^{1/r} \right\|_{L^r(\nu \otimes \nu_2 \otimes \cdots \otimes \nu_m)} = \left\| \left( \sum_{k_1, \dots, k_m=1}^n |T(f_{k_1}^1, \dots, f_{k_m}^m)|^r \right)^{1/r}\right\|_{L^r(\nu)}.
\end{equation}
Now, on the one hand we have
\begin{equation}\label{MZ for S and n=1}
 \left\| \left( \sum_{k_1=1}^n |S(f_{k_1}^1)|^r \right)^{1/r} \right\|_{L^r(\nu \otimes \nu_2 \otimes \cdots \otimes \nu_m)} \leq k_{q_1,r}^{(n)}(r) \|S\| \left\| \left( \sum_{k_1=1}^n |f_{k_1}^1| \right)^{1/r} \right\|_{L^{q_1}(\mu_1)}
\end{equation}
and on the other hand, if we call $T_{f^1}(f^2, \dots, f^m) = T(f^1,\dots, f^m)$ and $\q_{2, \dots, m}=(q_2, \dots, q_m)$, then
\begin{eqnarray*}
\|S(f^1)\|_{L^r(\nu \otimes \nu_2 \otimes \cdots \otimes \nu_m)}&=& \left( \int \sum_{k_2, \dots, k_m=1}^n |T_{f^1}(f_{k_2}^2, \dots, f_{k_m}^m)(\omega)|^r \, d\nu(\omega) \right)^{1/r}\\
&\leq& k_{\q_{2, \dots, m}, r}^{(n)}(r) \|T_{f^1}\| \prod_{i=2}^m \left\| \left( \sum_{k_i=1}^n |f_{k_i}^i| \right)^{1/r} \right\|_{L^{q_i}(\mu_i)}\\
&\leq&  k_{\q_{2, \dots, m}, r}^{(n)}(r) \|T\| \|f^1\|_{L^{q_1}(\mu_1)} \prod_{i=2}^m \left\| \left( \sum_{k_i=1}^n |f_{k_i}^i| \right)^{1/r} \right\|_{L^{q_i}(\mu_i)},
\end{eqnarray*}
from where we deduce
\begin{equation}\label{norma de S}
\|S\| \leq  k_{\q_{2, \dots, m}, r}^{(n)}(r) \|T\| \prod_{i=2}^m \left\| \left( \sum_{k_i=1}^n |f_{k_i}^i| \right)^{1/r} \right\|_{L^{q_i}(\mu_i)}.
\end{equation}
Putting  \eqref{relacion S y T}, \eqref{MZ for S and n=1} and \eqref{norma de S} together we obtain $k_{\q,r}^{(n)}(r) \leq k_{q_1,r}^{(n)}(r) k_{\q_{2, \dots, m}, r}^{(n)}(r)$. By the induction hypothesis we have $k_{\q_{2, \dots, m}, r}^{(n)}(r) \leq k_{q_2,r}^{(n)}(r) \cdots k_{q_m,r}^{(n)}(r)$ and this proves the statement.
\end{proof}

If $T_1\colon L^{q_1}(\mu_1) \to L^p(\nu_1)$ and $T_2\colon L^{q_2}(\mu_2) \times \cdots \times L^{q_m}(\mu_m) \to L^p(\nu_2)$, then $$T(f^1, \cdots, f^m)(\omega_1, \omega_2) = T_1(f^1)(\omega_1) T_2(f^2, \dots, f^m)(\omega_2)$$ is an $m$-linear operator satisfying $\|T\|=\|T_1\| \|T_2\|$ and an equality analogous to \eqref{relacion normas p}. Then, reasoning as in the proof of Proposition~\ref{relacion lineal-multilineal}, we see that if $(\q,p,r)$ satisfies the ($m$-linear) Marcinkiewicz-Zygmund inequality, then $((q_2,\dots,q_m),p,r)$ satisfies the ($(m-1)$-linear) Marcinkiewicz-Zygmund inequality. Clearly, the same reasoning  applies with any $q_i$ instead of $q_1$. Then, we have the following.
\begin{remark}
Let $1 \leq p,q_1, \dots, q_m, r \leq \infty$. Then,
$$
k_{q_j,p}(r) k_{(q_2, \dots,q_{j-1},q_{j+1},\dots,q_m),p}(r) \leq k_{(q_1, \dots, q_m),p}(r),
$$
for $j=1,\dots,m$.
\end{remark}
Applying Theorem~\ref{constantes kqp para todos los triples} and Proposition~\ref{relacion lineal-multilineal} for the case $p=r$, we see that $k_{\q,r}(r) < \infty$ if and only if $r=2$ or $q_i \leq r$ for all $1 \leq i \leq m$. Moreover, if $q_i \leq r$ for all $1 \leq i \leq m$ then $k_{\q,r}(r)=1$. This, together with the monotonicity in $p$ (see Proposition~\ref{monotonia en p y r}), gives the following result.
\begin{remark}\label{remark after monotonicity}
If $1 \leq q_1, \dots, q_m \leq r \leq p \leq \infty$ then $k_{\q,p}(r) =1$.
\end{remark}

\subsection{Monotonicity in $p$}
Next, we prove the property of monotonicity of $k_{\q,p}(r)$ as a function of $p$ stated in Proposition~\ref{monotonia en p y r} (i). The proof follows a duality argument analogous to that of \cite[Thm. 6]{GraMar}.

\begin{proof}[Proof of monotonicity in $p$]
Let $T \colon L^{q_1}(\mu_1) \times \cdots \times L^{q_m}(\mu_m) \to L^{p_1}(\nu)$ be an $m$-linear operator and $\{f_{k_i}^i\}_{k_i=1}^n \subset L^{q_i}(\mu_i)$. Assume $1 \leq r < \infty$, the case $r=\infty$ being completely analogous.
By duality we have
\begin{eqnarray}\label{dualidad}
&&\left\| \left(\sum_{k_1, \dots, k_m=1}^n |T(f_{k_1}^1, \dots, f_{k_m}^m)|^r\right)^{1/r} \right\|_{L^{p_1}(\nu)}\\
\nonumber &=& \sup_{\|g\|_{L^{(p_1/p_2)'}} \leq 1} \left\| \left(\sum_{k_1, \dots, k_m=1}^n |T(f_{k_1}^1, \dots, f_{k_m}^m)|^r\right)^{1/r} |g|^{1/p_2} \right\|_{L^{p_2}(\nu)}.
\end{eqnarray}
Now, fixed $g \in L^{(p_1/p_2)'}(\nu)$ with norm at most 1, consider $T_g \colon L^{q_1}(\mu_1) \times \cdots \times L^{q_m}(\mu_m) \to L^{p_2}(\nu)$ defined by
$$
T_g(f^1, \dots, f^m) = |g|^{1/p_2} T(f^1, \dots, f^m)
$$
and note that \eqref{dualidad} gives
$$
\left\| \left(\sum_{k_1, \dots, k_m=1}^n |T(f_{k_1}^1, \dots, f_{k_m}^m)|^r\right)^{1/r} \right\|_{L^{p_1}(\nu)} = \sup_{\|g\|_{L^{(p_1/p_2)'}} \leq 1} \left\| \left(\sum_{k_1, \dots, k_m=1}^n |T_g(f_{k_1}^1, \dots, f_{k_m}^m)|^r\right)^{1/r} \right\|_{L^{p_2}(\nu)}.
$$
Since $\|T_g\| \leq \|T\|$ for every $g \in L^{(p_1/p_2)'}(\nu)$ of norm at most 1  (just apply H$\ddot{\text{o}}$lder's inequality), we obtain
$$
\left\| \left(\sum_{k_1, \dots, k_m=1}^n |T(f_{k_1}^1, \dots, f_{k_m}^m)|^r\right)^{1/r} \right\|_{L^{p_1}(\nu)} \leq k_{\q,p_2}^{(n)}(r) \|T\| \prod_{i=1}^m \left\| \left(\sum_{k_i=1}^n |f_{k_i}^i|^r\right)^{1/r} \right\|_{L^{q_i}(\mu_i)},
$$
which gives $k_{\q,p_1}^{(n)}(r) \leq k_{\q,p_2}^{(n)}(r)$.
\end{proof}

\subsection{Monotonicity and interpolation property in $r$}
For the proof of the monotonicity of $k_{\q,p}(r)$ as a function of $r$ (when $1\leq r \leq 2$), we make use of $r$-stable variables (or $r$-stable L\'evy measures). A random variable is said to be $r$-stable ($0< r\leq 2$) if its Fourier transform on $\R$ is equal to $e^{-|x|^r}$. The $s$-th moment of the $r$-stable variable $w$ is given by $c_{r,s}=\left(\int_{\R} |t|^s w(t) dt \right)^{1/s}$. Note that 2-stable variables are just Gaussian variables.  One importance of $r$-stable variables relies on the following well known property, whose proof can be found in \cite[21.1.3]{Pie}:
For $0<s<r<2$ or $r=2$ and $0<s<\infty$, if $\{w_k\}_{k}$ is a sequence of independent $r$-stable random variables defined on $[0,1]$, then
\begin{equation}\label{ecuacion r estables}
\left(\int_0^1 \left| \sum_{k=1}^n a_k w_k(t) \right|^s dt \right)^{1/s} = c_{r,s} \left( \sum_{k=1}^n |a_k|^r \right)^{1/r}
\end{equation}
for every sequence $\{a_k\}_k\subset \mathbb R$ and any $n \in \N$.

The following lemma will be a main tool to prove the desired monotonicity and will also be  useful  in Section~\ref{caso menores que 2}.

\begin{lemma} \label{lemma r-stable}
Let $0<r\leq 2$ and $0<p<\infty$. If $\{w_{k_1}\}_{k_1}, \dots, \{w_{k_m}\}_{k_m}$ are sequences of mutually independent $r$-stable random variables defined on  $[0,1]$, then
\begin{eqnarray}\label{desigualdad r estables}
&& C \left( \sum_{k_1, \dots, k_m=1}^n |a_{k_1, \dots, k_m}|^r \right)^{1/r}\\
\nonumber& \leq&  \left(\int_0^1 \cdots \int_0^1 \left| \sum_{k_1, \dots, k_m=1}^n a_{k_1, \dots, k_m} w_{k_1}(t_1) \cdots w_{k_m}(t_m) \right|^p dt_1 \cdots dt_m \right)^{1/p}
\end{eqnarray}
for every sequence $\{a_{k_1, \dots, k_m}\}_{k_1, \dots, k_m}\subset \mathbb R$ and any $n \in \N$, where $C=c_{r,p}^{m}$ if $0<p < r <2$ or $0<p\leq r=2$,
$C=c_{2,2}^{m}$ if $r=2 < p < \infty$ and $C=c_{r, s}^m$ if $r<2$ and $r\leq p$, for any $0<s<r$.
\end{lemma}

\begin{proof}
We begin with the cases  $0<p < r <2$ and $0<p\leq r=2$. Note that, by \eqref{ecuacion r estables}, we have
\begin{eqnarray}\label{1er igualdad lema}
\nonumber (\star)&:=& \left( \sum_{k_1, \dots, k_m} |a_{k_1, \dots, k_m}|^r \right)^{1/r} = \left( \sum_{k_1, \dots, k_{m-1}} \left(\sum_{k_m} |a_{k_1, \dots, k_m}|^r \right)^{\frac{1}{r} r} \right)^{1/r} \\
&=& c_{r,p}^{-1} \left( \sum_{k_1, \dots, k_{m-1}} \left( \int_0^1 \left| \sum_{k_m} a_{k_1, \dots, k_m} w_{k_m}(t_m) \right|^p dt_m \right)^{r/p}  \right)^{1/r}  .
\end{eqnarray}
Then, applying the continuous Minkowski inequality with $r/p\geq1$ we obtain
$$
\sum_{k_{m-1}}  \left( \int_0^1 \left| \sum_{k_m} a_{k_1, \dots, k_m} w_{k_m}(t_m) \right|^p dt_m \right)^{r/p} \leq  \left( \int_0^1 \left( \sum_{k_{m-1}} \left| \sum_{k_m} a_{k_1, \dots, k_m} w_{k_m}(t_m) \right|^{r} \right)^{p/r} dt_m \right)^{r/p},
$$
which together with \eqref{1er igualdad lema} gives
\begin{equation}\label{2da desigualdad lema}
(\star) \leq c_{r,p}^{-1}  \left( \sum_{k_1, \dots, k_{m-2}} \left( \int_0^1 \left( \sum_{k_{m-1}} \left| \sum_{k_m} a_{k_1, \dots, k_m} w_{k_m}(t_m) \right|^{r} \right)^{p/r} dt_m \right)^{r/p} \right)^{1/r}.
\end{equation}
Using  \eqref{ecuacion r estables} again, we see that
$$
 \left( \sum_{k_{m-1}} \left| \sum_{k_m} a_{k_1, \dots, k_m} w_{k_m}(t_m) \right|^{r} \right)^{p/r} = c_{r,p}^{-p} \int_0^1 \left| \sum_{k_{m-1}} \sum_{k_m} a_{k_1, \dots, k_m} w_{k_m}(t_m) w_{k_{m-1}}(t_{m-1})  \right|^p dt_{m-1}
$$
and, in virtue of \eqref{2da desigualdad lema},
$$
(\star) \leq c_{r,p}^{-2} \left( \sum_{k_1, \dots, k_{m-2}} \left( \int_0^1 \int_0^1  \left| \sum_{k_{m-1}, k_m} a_{k_1, \dots, k_m} w_{k_m}(t_m) w_{k_{m-1}}(t_{m-1})  \right|^p dt_{m-1} dt_m  \right)^{r/p} \right)^{1/r}.
$$
Repeating this argument, in $(m-1)$-steps we obtain
$$
(\star) \leq c_{r,p}^{-(m-1)} \left( \sum_{k_1} \left( \int_0^1 \cdots \int_0^1 \left| \sum_{k_2, \dots, k_m} a_{k_1, \dots, k_m} w_{k_m}(t_m) \cdots w_{k_{2}}(t_{2})  \right|^p dt_2 \dots dt_m \right)^{r/p} \right)^{1/r}
$$
and applying Minkowski's inequality (with $r/p\geq 1$) and \eqref{ecuacion r estables} we get
$$
(\star) \leq c_{r,p}^{-m}  \left(\int_0^1 \cdots \int_0^1 \left| \sum_{k_1, \dots, k_m} a_{k_1, \dots, k_m} w_{k_1}(t_1) \cdots w_{k_m}(t_m) \right|^p dt_1 \cdots dt_m \right)^{1/p},
$$
which proves the statement.

The remaining cases, where $p$ is greater than or equal to $r$, follow choosing  $s=r=2$  or $s<r<2$ and applying the previous cases to get
\begin{eqnarray*}
c_{r,s}^{m} \left( \sum_{k_1, \dots, k_m} |a_{k_1, \dots, k_m}|^r \right)^{1/r}
 &\le & \left\|  \sum_{k_1, \dots, k_m} a_{k_1, \dots, k_m} w_{k_1} \cdots w_{k_m}  \right\|_{L^s([0,1]^m)} \\
& \le &  \left\|  \sum_{k_1, \dots, k_m} a_{k_1, \dots, k_m} w_{k_1} \cdots w_{k_m}  \right\|_{L^p([0,1]^m)}. \qedhere
\end{eqnarray*}
\end{proof}

Another useful tool in the proof of the monotonicity in $r$, is the following  generalization to the multilinear setting of a result that can be found inside the proof of \cite[Thm. 1]{GasMal}.
\begin{lemma}\label{lema monotonia en r}
Let $1 \leq q_1, \dots, q_m, p\leq \infty$ and $1 \leq r < \infty$. Then
\begin{eqnarray}\label{desigualdad lema monotonia en r}
\left\| \left( \int_0^1 \cdots \int_0^1 \left| \sum_{k_1, \dots, k_m=1}^n T(f_{k_1}^1, \dots, f_{k_m}^m) g_{k_1}^1(t_1) \cdots g_{k_m}^m(t_m) \right|^r dt_1 \dots dt_m \right)^{1/r} \right\|_{L^p(\nu)} \\
\nonumber\leq k_{\q,p}(r) \|T\| \prod_{i=1}^m \left\| \left( \int_0^1 \left| \sum_{k_i=1}^n f_{k_i}^i g_{k_i}^i(t_i) \right|^r dt_i \right)^{1/r} \right\|_{L^{q_i}(\mu_i)}
\end{eqnarray}
for each bounded multilinear operator $T\colon L^{q_1}(\mu_1) \times \cdots \times L^{q_m}(\mu_m) \to L^p(\nu)$,
functions $\{f_{k_i}^i\}_{k_i=1}^{n} \subset L^{q_i}(\mu_i)$, $\{g_{k_i}^i\}_{k_i=1}^n \subset L^r[0,1]$ ($i=1, \dots, m$) and $n \in \N$.
\end{lemma}

\begin{proof}
Fix $T\colon L^{q_1}(\mu_1) \times \cdots \times L^{q_m}(\mu_m) \to L^p(\nu)$, $n \in \N$ and $\{f_{k_i}^i\}_{k_i=1}^{n} \subset L^{q_i}(\mu_i)$, $\{g_{k_i}^i\}_{k_i=1}^n \subset L^r[0,1]$. Along the proof, we will write $\|\cdot\|_r$ instead of $\|\cdot\|_{L^r[0,1]}$ or $\|\cdot\|_{L^r[0,1]^m}$, depending on the context. Let $\varepsilon >0$ and, for each $1 \leq i \leq m$, consider a sequence of simple functions $\{s_{k_i}^i\}_{k_i=1}^n$ such that $\|g_{k_i}^i - s_{k_i}^i\|_{r} < \varepsilon$, $k_i=1, \dots, n$. On the one hand, reasoning as in  \cite[Thm. 1]{GasMal}, we obtain
\begin{equation}\label{cota simples}
\left\| \left\| \sum_{k_i=1}^n f_{k_i}^i s_{k_i}^i \right\|_r\right\|_{q_i} \leq \left\| \left\| \sum_{k_i=1}^n f_{k_i}^i g_{k_i}^i \right\|_r\right\|_{q_i} + \varepsilon \left\| \sum_{k_i=1}^n |f_{k_i}^i|\right\|_{q_i}  \leq \left\| \left\| \sum_{k_i=1}^n f_{k_i}^i g_{k_i}^i \right\|_r\right\|_{q_i} + \varepsilon \sum_{k_i=1}^n \|f_{k_i}^i\|_{q_i}.
\end{equation}
On the other hand, $\|g_{k_1}^1 \cdots g_{k_m}^m - s_{k_1}^1 \cdots s_{k_m}^m\|_{r} < \gamma(\varepsilon)$ for some $\gamma(\varepsilon) \xrightarrow[\varepsilon \to 0]{}\, 0$ (just add and subtract the terms $s_{k_1}^1 g_{k_2}^2 \cdots g_{k_m}^m, s_{k_1}^1 s_{k_2}^2 \cdots g_{k_m}^m, \dots, s_{k_1}^1 \cdots s_{k_{m-1}}^{m-1} g_{k_m}^m$ and apply triangle inequality). Then, we have
\begin{eqnarray*}
&& \left| \left\| \sum_{k_1, \dots, k_m=1}^n T(f_{k_1}^1, \dots, f_{k_m}^m)(\omega) g_{k_1}^1 \cdots g_{k_m}^m \right\|_{r} - \left\| \sum_{k_1, \dots, k_m=1}^n T(f_{k_1}^1, \dots, f_{k_m}^m)(\omega) s_{k_1}^1 \cdots s_{k_m}^m \right\|_{r} \right| \\
&\leq&  \sum_{k_1, \dots, k_m=1}^n |T(f_{k_1}^1, \dots, f_{k_m}^m)(\omega)| \left\|g_{k_1}^1 \cdots g_{k_m}^m - s_{k_1}^1 \cdots s_{k_m}^m\right\|_r
\leq \gamma(\varepsilon) \sum_{k_1, \dots, k_m=1}^n |T(f_{k_1}^1, \dots, f_{k_m}^m)(\omega)|.
\end{eqnarray*}
Note that $\gamma(\varepsilon)$ is independent of the $k_i$'s; it depends on $\max_{i, k_i} \|g_{k_i}^i\|_{r}$ but this is not a problem since the functions $\{g_{k_i}^i\}_{k_i}$ are fixed.
Now, the previous inequality together
with the monotonicity of the norm $\|\cdot\|_p$, gives
\begin{eqnarray}\label{monotonia en p y desig triangular}
\nonumber (\dag):=\left\| \left\| \sum_{k_1, \dots, k_m=1}^n T(f_{k_1}^1, \dots, f_{k_m}^m) g_{k_1}^1 \cdots g_{k_m}^m \right\|_{r} \right\|_p \leq \left\| \left\| \sum_{k_1, \dots, k_m=1}^n T(f_{k_1}^1, \dots, f_{k_m}^m) s_{k_1}^1 \cdots s_{k_m}^m \right\|_{r} \right\|_p\\
+ \gamma(\varepsilon) \left\|  \sum_{k_1, \dots, k_m=1}^n |T(f_{k_1}^1, \dots, f_{k_m}^m)| \right\|_p.
\end{eqnarray}
Then, if we prove that
\begin{equation} \label{simples kqp}
\left\| \left\| \sum_{k_1, \dots, k_m=1}^n T(f_{k_1}^1, \dots, f_{k_m}^m) s_{k_1}^1 \cdots s_{k_m}^m \right\|_{r} \right\|_p
\leq k_{\q,p}(r) \|T\| \prod_{i=1}^m \left\| \left\| \sum_{k_i=1}^n f_{k_i}^i s_{k_i}^i \right\|_r \right\|_{q_i},
\end{equation}
we would have
\begin{eqnarray*}
&&(\dag) \leq k_{\q,p}(r) \|T\| \prod_{i=1}^m \left\| \left\| \sum_{k_i=1}^n f_{k_i}^i s_{k_i}^i \right\|_r \right\|_{q_i}  + \gamma(\varepsilon) \left\|  \sum_{k_1, \dots, k_m=1}^n |T(f_{k_1}^1, \dots, f_{k_m}^m)| \right\|_p \quad \text{(by \eqref{monotonia en p y desig triangular} and \eqref{simples kqp})}\\
&\leq&  k_{\q,p}(r) \|T\| \prod_{i=1}^m \left(\left\| \left\| \sum_{k_i=1}^n f_{k_i}^i g_{k_i}^i \right\|_r\right\|_{q_i} + \varepsilon \sum_{k_i=1}^n \|f_{k_i}^i\|_{q_i}\right) + \gamma(\varepsilon) \|T\| \sum_{k_1, \dots, k_m=1}^n \prod_{i=1}^m\|f_{k_i}^i\|_{q_i}, \quad \text{(by \eqref{cota simples})}
\end{eqnarray*}
and since $\varepsilon >0$ was chosen arbitrarily, letting $\varepsilon \to 0$,
$$
(\dag) \leq k_{\q,p}(r) \|T\| \prod_{i=1}^m \left\| \left\| \sum_{k_i=1}^n f_{k_i}^i g_{k_i}^i \right\|_r\right\|_{q_i}
$$
which is the desired statement. Then, it only remains to prove \eqref{simples kqp}. Since each $s_{k_j}^j$ is a simple function, there exist $c_{i_j, k_j}^j \in \R$ and $A_{i_j}$ measurable (disjoint) subsets of $[0,1]$ such that
$
s_{k_j}^j(\cdot) = \sum_{i_j}  c_{i_j, k_j}^j  \chi_{A_{i_j}}(\cdot).
$
Now, if we denote by $\lambda(A_{i_j})$ the measure of $A_{i_j}$,
\begin{eqnarray*}
&&\left\| \left\| \sum_{k_1, \dots, k_m=1}^n T(f_{k_1}^1, \dots, f_{k_m}^m) s_{k_1}^1 \cdots s_{k_m}^m \right\|_{r} \right\|_p \\
&=& \left( \int \left( \int_0^1 \cdots \int_0^1 \left| \sum_{k_1, \dots, k_m=1}^n T(f_{k_1}^1, \dots, f_{k_m}^m)(\omega) s_{k_1}^1(t_1) \cdots s_{k_m}^m(t_m) \right|^r dt_1 \cdots dt_m \right)^{p/r} d\nu(\omega) \right)^{1/p}\\
&=& \left( \int \left( \sum_{i_1, \dots, i_m} \left| \sum_{k_1, \dots, k_m=1}^n T(f_{k_1}^1, \dots, f_{k_m}^m)(\omega) c_ {i_1,k_1}^1 \dots c_{i_m,k_m}^m \lambda(A_{i_1})^{1/r} \cdots \lambda(A_{i_m})^{1/r} \right|^r \right)^{p/r} d\nu(\omega) \right)^{1/p}\\
&=& \left( \int \left( \sum_{i_1, \dots, i_m} \left| T\left(\sum_{k_1=1}^n c_{i_1, k_1}^1 \lambda(A_{i_1})^{1/r} f_{k_1}^1, \dots, \sum_{k_m=1}^n c_{i_m, k_m}^m \lambda(A_{i_m})^{1/r} f_{k_m}^m \right)(\omega) \right|^r  \right)^{p/r} d\nu(\omega) \right)^{1/p}\\
&\leq& k_{\q,p}(r) \|T\| \prod_{i=1}^m \left\| \left\| \sum_{k_i=1}^n f_{k_i}^i s_{k_i}^i \right\|_r \right\|_{q_i}
\end{eqnarray*}
which proves \eqref{simples kqp} (once again, we are assuming  $1 \leq p < \infty$, the case $p=\infty$  being completely analogous).
\end{proof}

Finally, we prove the monotonicity in $r$ stated in Proposition~\ref{monotonia en p y r} (ii).

\begin{proof}[Proof of monotonicity in $r$]
Let $1 \leq r<s \leq 2$. The proof follows by a simple application of Lemma~\ref{lemma r-stable} and  Lemma~\ref{lema monotonia en r} with the sequences $\{g_{k_i}^i\}_{k_i}=\{w_{k_i}\}_{k_i}$ of $s$-stable random variables and  functions $\{f_{k_i}^i\}_{k_i=1}^{n} \subset L^{q_i}(\mu_i)$. Indeed, we have
\begin{eqnarray*}
&&c_{s,r}^m \left\| \left( \sum_{k_1, \dots, k_m} |T(f_{k_1}^1, \dots, f_{k_m}^m)| \right)^{1/s} \right\|_{L^p(\nu)}\\
 &\leq& \left\| \left( \int_0^1 \cdots \int_0^1 \left|  \sum_{k_1, \dots, k_m} T(f_{k_1}^1, \dots, f_{k_m}^m) w_{k_1}(t_1) \cdots w_{k_m}(t_m) \right|^r dt_1 \dots dt_m \right)^{1/r} \right\|_{L^p(\nu)} \quad \text{(by \eqref{desigualdad r estables})}\\
 &\leq& k_{\q,p}(r) \|T\| \prod_{i=1}^m \left\| \left( \int_0^1 \left| \sum_{k_i=1}^n f_{k_i}^i w_{k_i}(t_i) \right|^r dt_i \right)^{1/r} \right\|_{L^{q_i}(\mu_i)} \quad \text{(by \eqref{desigualdad lema monotonia en r})}\\
&=& k_{\q,p}(r) \|T\| c_{s,r}^m \prod_{i=1}^m \left\| \left( \sum_{k_i=1}^n |f_{k_i}^i|^s \right)^{1/s} \right\|_{L^{q_i}(\mu_i)} \quad \text{(by \eqref{ecuacion r estables})}
\end{eqnarray*}
and this proves that $k_{\q,p}^{(n)}(s) \leq k_{\q,p}(r)$. Consequently, $k_{\q,p}(s) \leq k_{\q,p}(r)$.
\end{proof}

It is worth mentioning that the above proof does not assure $k_{\q,p}^{(n)}(s) \leq k_{\q,p}^{(n)}(r)$ for each $n \in \N$. The problem arises in Lemma~\ref{lema monotonia en r}, where we cannot put $k_{\q,p}^{(n)}(r)$ instead of  $k_{\q,p}(r)$. This problem was stated in \cite[Problem 1]{GasMal} in the linear case.

We will see now a kind of \textit{interpolation} behavior of $k_{\q,p}(r)$ as a function of $r$. First, we need the following known result (see, for instance, \cite[Sections 5.6, 5.7 and 5.8]{BerLof}).

\noindent\emph{
Let $1 \leq r_1,r_2 \leq \infty$ and $T$ be a multilinear operator which is bounded from $L^{q_1}(\mu_1, \ell^{r_i}) \times \cdots \times L^{q_m}(\mu_m, \ell^{r_i})$ into $L^p(\nu, \ell^{r_i}(\N \times \cdots \times \N))$ with norm $M_i$ ($i=1,2$). If $\frac{1}{r}=\frac{1-\theta}{r_1} + \frac{\theta}{r_2}$ with $0<\theta <1$,  then $T \colon L^{q_1}(\mu_1, \ell^{r}) \times \cdots \times L^{q_m}(\mu_m, \ell^{r}) \to L^p(\nu, \ell^{r}(\N \times \cdots \times \N))$ with norm less than or equal to $M_1^{1 - \theta}M_2^{\theta}$. The same is true if we put $\ell^r_n$ instead of $\ell^r$, where $\ell^r_n(\N \times \cdots \times \N)$ is the space of sequences $(a_{k_1, \dots, k_m})_{k_1, \dots, k_m=1}^n$ with the norm $r$.
}

As a consequence, we obtain the mentioned property of $k_{\q,p}(r)$.

\begin{corollary}\label{interpolation property of r}
Let $1 \leq p, q_1, \dots, q_m, r_1, r_2 \leq \infty$. If $\frac{1}{r}=\frac{1-\theta}{r_1} + \frac{\theta}{r_2}$ with $0<\theta <1$,  then $k_{\q,p}^{(n)}(r) \leq k_{\q,p}^{(n)}(r_1)^{1-\theta} k_{\q,p}^{(n)}(r_2)^{\theta}$. Consequently, if $k_{\q,p}(r_i)< \infty$ for $i=1,2$, then  $k_{\q,p}(r) \leq k_{\q,p}(r_1)^{1-\theta} k_{\q,p}(r_2)^{\theta} < \infty$.
\end{corollary}

\begin{proof}
It is enough to prove that, given $T \colon L^{q_1}(\mu_1) \times \cdots \times L^{q_m}(\mu_m) \to L^p(\nu)$ with norm $\|T\|\leq 1$, the $\ell^r_n$-valued extension $T^{\ell^r_n} \colon  L^{q_1}(\mu_1, \ell^{r}_n) \times \cdots \times L^{q_m}(\mu_m, \ell^{r}_n) \to L^p(\nu, \ell^{r}_n(\N \times \cdots \times \N))$ defined as in \eqref{l_r valued multilinear operator}  has norm less than or equal to $k_{\q,p}^{(n)}(r_1)^{1-\theta} k_{\q,p}^{(n)}(r_2)^{\theta}$. For this, simply note that each $\ell^{r_i}_n$-valued extension $T^{\ell^{r_i}_n}$ ($i=1,2$) has norm less than or equal to $k_{\q,p}^{(n)}(r_i)$ and apply the previous interpolation theorem.
\end{proof}

\subsection{The case $\q=(1, \dots, 1)$}
To finish this section we prove a generalization to the multilinear setting of a result stated in Theorem~\ref{constantes kqp para todos los triples} (i), which asserts that $k_{1,p}(r) =1$ for all $1 \leq p,r \leq \infty$. Recall that $L^p(\mu, X)$ is the space of $p$-integrable $X$-valued functions.
When $p=1$, there is a natural (isometric) identification between the projective tensor product $L^1(\mu) \tilde{\otimes}_{\pi} X$ and the space $L^1(\mu, X)$. Having in mind this isometry, we adopt the tensor notation $g\otimes x(\omega) = g(\omega) x$.
We refer readers to \cite[3.3]{DefFlo} for the definition of the projective tensor product and a detailed exposition of this topics. For our purposes, we will need the following well-known fact: given $f \in L^1(\mu,X)$ (or $L^1(\mu) \tilde{\otimes}_{\pi} X$, via the identification) and $\varepsilon>0$ there exist bounded sequences $(g_k)_k \subset L^1(\mu)$ and $(x_k)_k \subset X$ such that  the series $\sum_{k=1}^\infty g_k \otimes x_k$ converges to $f$ (in the projective norm) and
$
\sum_{k=1}^\infty \|g_k\|_{L^1(\mu)} \|x_k\|_X < \|f\|_{L^1(\mu, X)} + \varepsilon.
$
Then, it is clear that
\begin{equation}\label{L1 tensorial}
\|f\|_{L^1(\mu, X)} = \inf \left\{\sum_{k} \|g_k\|_{L^1(\mu)} \|x_k\|_X\right\},
\end{equation}
where the infimum is taken over all the representations $f=\sum_{k=1}^\infty g_k \otimes x_k$ with $g_k \in L^1(\mu)$ and $x_k \in X$.

\begin{proposition}\label{obs q=1}
 If $\q=(1, \dots, 1)$, then $k_{\q,p}(r)=1$ for all $1 \leq p,r \leq \infty$.
\end{proposition}

\begin{proof}
We suppose $1 \leq p, r < \infty$ (the proof is the same in the remaining cases).
Let $T\colon L^1(\mu_1) \times \cdots \times L^1(\mu_m) \to L^p(\nu)$ and functions $\{f_{k_i}^i\}_{k_i=1}^n \in L^1(\mu_i)$, $i=1, \dots, m$. Consider $h_i \in L^1(\mu_i, \ell^r)$ defined by $h_i=\sum_{k_i=1}^n f_{k_i}^i \otimes e_{k_i}$ and note that
\begin{equation}\label{norma hi}
\|h_i\|_{L^1(\mu_i, \ell^r)} = \int \left\| \sum_{k_i=1}^n f_{k_i}^i(\omega) e_{k_i} \right\|_{\ell^r} d\mu_i(\omega)= \left\| \left( \sum_{k_i=1}^n |f_{k_i}^i|^r \right)^{1/r} \right\|_{L^1(\mu_i)}.
\end{equation}
By \eqref{L1 tensorial}, given $\varepsilon >0$ we can take (for each $1\leq i \leq m$) a representation $h_i= \sum_{l_i} g_{l_i}^i \otimes x_{l_i}^i$ such that
\begin{equation} \label{una representacion}
\sum_{l_i} \|g_{l_i}^i\|_{L^1(\mu_i)} \|x_{l_i}^i\|_{\ell^r} < \|h_i\|_{L^1(\mu_i, \ell^r)} + \varepsilon.
\end{equation}
Now, consider $T^{\ell^r} \colon L^1(\mu_1, \ell^r) \times \cdots \times L^1(\mu_m, \ell^r) \to L^p(\nu, \ell^r(\N \times \cdots \times \N))$ the $m$-linear operator defined by $T^{\ell^r}(f^1 \otimes x^1, \dots, f^m \otimes x^m)(\cdot) = T(f^1, \dots, f^m)(\cdot) x^1\otimes \cdots \otimes x^m$, where ${x^1\otimes \cdots \otimes x^m}$ stands for $(x^1(i_1) \cdots x^m(i_m))_{i_1, \dots, i_m=1}^\infty \in \ell^r(\N \times \cdots \times \N)$. On the one hand we have that $T^{\ell^r}(h_1, \dots, h_m)(\cdot)=\sum_{k_1, \dots,k_m=1}^n T(f_{k_1}^1, \dots, f_{k_m}^m)(\cdot) e_{k_1}\otimes \cdots \otimes e_{k_m}$ and, hence,
\begin{eqnarray}\label{norma Ttilde}
\nonumber\|T^{\ell^r}(h_1, \dots, h_m)\|_{L^p(\nu, \ell^r(\N \times \cdots \times \N))} &=& \left(\int \left\| \sum_{k_1, \dots, k_m=1}^n T(f_{k_1}^1, \dots, f_{k_m}^m)(\omega) e_{k_1}\otimes \cdots \otimes e_{k_m}  \right\|_{\ell^r}^p d\nu(\omega) \right)^{1/p}\\
&=& \left(\int \left( \sum_{k_1, \dots, k_m=1}^n |T(f_{k_1}^1, \dots, f_{k_m}^m)(\omega)|^r \right)^{p/r} d\nu(\omega) \right)^{1/p}.
\end{eqnarray}
On the other hand
$
T^{\ell^r}(h_1, \dots, h_m)(\cdot)=\sum_{l_1, \dots, l_m} T(g_{l_1}^1, \dots, g_{l_m}^m)(\cdot) x_{l_1}^1 \otimes \cdots \otimes x_{l_m}^m,
$
from where
\begin{eqnarray}\label{norma Ttilde2}
\nonumber && \|T^{\ell^r}(h_1, \dots, h_m)\|_{L^p(\nu, \ell^r(\N \times \cdots \times \N))} = \left\| \sum_{l_1, \dots, l_m} T(g_{l_1}^1, \dots, g_{l_m}^m)(\cdot) x_{l_1}^1 \otimes \cdots \otimes x_{l_m}^m  \right\|_{L^p(\nu, \ell^r(\N \times \cdots \times \N))} \\
\nonumber &\leq& \sum_{l_1, \dots, l_m} \|T(g_{l_1}^1, \dots, g_{l_m}^m)\|_{L^p(\nu)} \|x_{l_1}^1\|_{\ell^r} \cdots \|x_{l_m}^m\|_{\ell^r}
\leq \|T\| \sum_{l_1, \dots, l_m} \left(\prod_{i=1}^m\|g_{l_i}^i\|_{L^1(\mu_i)} \right)  \left(\prod_{i=1}^m \|x_{l_i}^i\|_{\ell^r}\right) \\
\nonumber&=& \|T\| \prod_{i=1}^m \left(\sum_{l_i}  \|g_{l_i}^i\|_{L^1(\mu_i)} \|x_{l_i}^i\|_{\ell^r}  \right)
\leq \|T\| \prod_{i=1}^m \left(\left\| \left( \sum_{k_i=1}^n |f_{k_i}^i|^r \right)^{1/r} \right\|_{L^1(\mu_i)}+\varepsilon\right),
\end{eqnarray}
the last inequality due to \eqref{norma hi} and \eqref{una representacion}. Putting together \eqref{norma Ttilde} and the inequality above we obtain
$$
 \left(\int \left( \sum_{k_1, \dots, k_m=1}^n |T(f_{k_1}^1, \dots, f_{k_m}^m)(\omega)|^r \right)^{p/r} d\nu(\omega) \right)^{1/p} \leq  \|T\| \prod_{i=1}^m \left(\left\| \left( \sum_{k_i=1}^n |f_{k_i}^i|^r \right)^{1/r} \right\|_{L^1(\mu_i)}+\varepsilon\right)
$$
and since $\varepsilon >0$ was arbitrary we deduce that $k_{(1, \dots,1), p}^{(n)}(r) =1$ for all $n \in \N$.
\end{proof}

\section{The triples $(\q,p,r)$ satisfying $k_{\q,p}(r) < \infty$}\label{section main results}

\subsection{The case $1 \leq q_1,\dots, q_m \leq 2$, $p<\infty$} \label{caso menores que 2}
In this section we see that, in the particular case  $1 \leq q_1,\dots, q_m \leq 2$ and $1 \leq p<\infty$, we can determine the triples $(p,\q,r)$ satisfying the Marcinkiewicz-Zygmund inequality and the exact values of the constants $k_{\q,p}(r)$. We follow ideas of \cite{GasMal} (see Theorems~2 and 3 in there) to obtain upper estimates for the constants $k_{\q,p}(r)$ and then use Proposition~\ref{relacion lineal-multilineal} for the lower bounds. The exact value of $k_{\q,p}(r)$ is then determined by the exact values of the constants $k_{q_i,p}(r)$ obtained in \cite{DefJun}.

\begin{theorem}\label{caracterizacion caso menor 2}
Let $1\leq q_1, \dots, q_m \leq 2$, $1 \leq p < \infty$ and denote  $\qm = \max\{q_1, \dots, q_m\}$. Then $k_{\q,p}(r) < \infty$ if and only if one of the following conditions is satisfied:
\begin{enumerate}
\item[\rm (i)] $r=2$;
\item[\rm (ii)]  $\qm =1$ and $1 \leq r \leq \infty$;
\item[\rm (iii)] $\qm < r \leq \max\{p,2\}$;
\item[\rm (iv)] $\qm =r \leq p$.
\end{enumerate}
Moreover, in all these cases we have $k_{\q,p}(r) = k_{q_1,p}(r) \cdots k_{q_m,p}(r)$.
\end{theorem}
It can be easily deduced from the preceding theorem and Theorem~\ref{constantes kqp para todos los triples} that, in the case $1 \leq q_1, \dots, q_m \leq 2$ and $1 \leq p < \infty$, we have $k_{\q,p}(r)<\infty$ if and only if $k_{q_i,p}(r) < \infty$ for all $1 \leq i\leq m$.
Then, in virtue of the monotonicity of $k_{q_i,p}(r)$ as a function of $q_i$  (see Theorem~\ref{propiedades de los kqp}), it follows that $k_{\q,p}(r) < \infty$ if and only if $k_{\qm,p}(r) < \infty$.
Note that, as pointed out in the Introduction (see Proposition~\ref{proposicion caso infinito} and the paragraph above), this equivalence does not hold in general, as we will see in Section~\ref{caso p infinito}.
\begin{proof}[Proof of Theorem~\ref{caracterizacion caso menor 2}]
Suppose first that $k_{\q,p}(r) < \infty$ and (ii) is not satisfied, and let us see that either (i), (iii) or (iv) is satisfied. By Proposition~\ref{relacion lineal-multilineal} we have $k_{q_i,p}(r) < \infty$ for all $1 \leq i \leq m$ and, in particular, $k_{\qm,p}(r) < \infty$.
If $\qm \leq p$, we have $1 \leq \qm \leq p < \infty$ and $k_{\qm,p}(r)<\infty$ and, by Theorem~\ref{constantes kqp para todos los triples}, it follows that $\qm \leq r \leq \max\{p,2\}$.
If  $\qm > p$, then we have $1\leq p < \qm\leq 2$ and $k_{\qm,p}(r)< \infty$ and by Theorem~\ref{constantes kqp para todos los triples} it follows that $r=2$ or $\qm<r\leq 2 = \max\{p,2\}$.
We deduce from the cases above that either $r=2$ or $\qm < r \leq \max\{p,2\}$ or $\qm =r$ and $\qm \leq p$.
For the converse, assume the following facts.
\begin{enumerate}
\item[\rm (a)] If $1 \leq p, \qm <r<2$ or $1 \leq p, \qm \leq r=2$, then $k_{\q,p}(r) = k_{q_1,p}(r)\cdots k_{q_m,p}(r) < \infty$.
\item[\rm (b)] If $\qm \leq r \leq p$, then $k_{\q,p}(r) = 1 = k_{q_1,p}(r)\cdots k_{q_m,p}(r)$.
\end{enumerate}
If (i) holds then, whether $p<2$ or $p\geq 2$, $k_{\q,p}(r) = k_{q_1,p}(r)\cdots k_{q_m,p}(r)<\infty$ in virtue of (a) and (b).
If (ii) holds then by Proposition~\ref{obs q=1} we have $k_{\q,p}(r) = 1 = k_{q_1,p}(r)\cdots k_{q_m,p}(r)$.
Suppose we are under the hypothesis of (iii). If $p<r$ then $\qm < r \leq \max\{p,2\}= 2$ and hence $k_{\q,p}(r) = k_{q_1,p}(r)\cdots k_{q_m,p}(r) < \infty$ by (a), while if $p \geq r$ then $\qm < r \leq p$ and, consequently, $k_{\q,p}(r) = 1 = k_{q_1,p}(r)\cdots k_{q_m,p}(r)$ by (b). Finally, if (iv) holds then $k_{\q,p}(r) = 1 = k_{q_1,p}(r)\cdots k_{q_m,p}(r)$ by (b). Then, it only remains to prove (a) since the assertion in (b) was already stated in Remark~\ref{remark after monotonicity}. Fix $T\colon L^{q_1}(\mu_1)\times \cdots \times L^{q_m}(\mu_m) \to L^p(\nu)$ and functions $\{f_{k_i}^i\}_{k_i=1}^n \subset L^{q_i}(\mu_i)$, $1\leq i \leq m$. By Lemma~\ref{lemma r-stable} and the $m$-linearity of $T$ we have
\begin{eqnarray*}
&&(\star)=c_{r,p}^{mp} \int \left( \sum_{k_1, \dots, k_m=1}^n |T(f_{k_1}^1, \dots, f_{k_m}^m)(\omega)|^r \right)^{p/r} d\nu(\omega) \\
&\leq& \int \int_0^1 \cdots \int_0 ^1 \left|\sum_{k_1, \dots, k_m=1}^n T(f_{k_1}^1, \dots, f_{k_m}^m)(\omega) w_{k_1}(t_1) \cdots w_{k_m}(t_m)\right|^p dt_1 \cdots dt_m d\nu(\omega)\\
&=& \int \int_0^1 \cdots \int_0 ^1 \left| T\left(\sum_{k_1=1}^n f_{k_1}^1 w_{k_1}(t_1), \dots, \sum_{k_m=1}^n f_{k_m}^m w_{k_m}(t_m) \right)(\omega) \right|^p dt_1 \cdots dt_m d\nu(\omega).
\end{eqnarray*}
Applying Fubini and the boundedness of $T$ we obtain,
\begin{eqnarray}
\nonumber(\star)&\leq& \int_0^1 \cdots \int_0 ^1 \left\| T\left(\sum_{k_1=1}^n f_{k_1}^1 w_{k_1}(t_1), \cdots, \sum_{k_m=1}^n f_{k_m}^m w_{k_m}(t_m) \right) \right\|_{L^p(\nu)}^p  dt_1 \cdots dt_m \\
\label{cota dag}&=& \|T\|^p \prod_{i=1}^m \int_0^1 \left\| \sum_{k_i=1}^n f_{k_i}^i w_{k_i}(t_i) \right\|^p_{L^{q_i}(\mu_i)}  dt_i.
\end{eqnarray}
Then, it will be suffice to obtain upper bounds for each $\int_0^1 \left\| \sum_{k_i=1}^n f_{k_i}^i w_{k_i}(t_i) \right\|^p_{L^{q_i}(\mu_i)}  dt_i$. On the one hand, for those $q_i<p$ we have
\begin{eqnarray*}
&& \int_0^1 \left\| \sum_{k_i=1}^n f_{k_i}^i w_{k_i}(t_i) \right\|^p_{L^{q_i}(\mu_i)}  dt_i = \left[ \left( \int_0^1 \left( \int \left| \sum_{k_i=1}^n f_{k_i}^i(\omega) w_{k_i}(t_i) \right|^{q_i} d\mu_i(\omega) \right)^{p/q_i} dt_i \right)^{q_i/p} \right]^{p/q_i}\\
&\leq& \left[ \int \left( \int_0^1 \left| \sum_{k_i=1}^n f_{k_i}^i(\omega) w_{k_i}(t_i) \right|^p dt_i \right)^{q_i/p}  d\mu_i(\omega) \right]^{p/q_i} \quad \text{(by Minkowski with $p/q_i > 1$)}\\
&=& c_{r,p}^{p} \left\| \left( \sum_{k_i=1}^n |f_{k_i}^i|^r \right)^{1/r} \right\|^{p}_{L^{q_i}(\mu_i)} \quad \text{(by \eqref{ecuacion r estables}).}
\end{eqnarray*}
On the other hand, if $p \leq q_i$
\begin{eqnarray*}
\int_0^1 \left\| \sum_{k_i=1}^n f_{k_i}^i w_{k_i}(t_i) \right\|^p_{L^{q_i}(\mu_i)}  dt_i &\leq& \left( \int_0 ^1 \left\| \sum_{k_i=1}^n f_{k_i}^i w_{k_i}(t_i) \right\|^{q_i}_{L^{q_i}(\mu_i)}  dt_i \right)^{p/q_i} \quad \text{(by H$\ddot{\text{o}}$lder with $q_i/p\geq 1$)}\\
&=& \left(  \int \int_0^1 \left| \sum_{k_i=1}^n f_{k_i}^i(\omega) w_{k_i}(t_i) \right|^{q_i}   dt_i d\mu_i(\omega) \right)^{p/q_i} \quad \text{(by Fubini)}\\
&=& c_{r,q_i}^{p} \left\| \left( \sum_{k_i=1}^n |f_{k_i}^i|^r \right)^{1/r} \right\|^{p}_{L^{q_i}(\mu_i)} \quad \text{(by \eqref{ecuacion r estables}).}
\end{eqnarray*}
These inequalities together with \eqref{cota dag} give
$$
\left\|\left( \sum_{k_1, \dots, k_m=1}^n |T(f_{k_1}^1, \dots, f_{k_m}^m)|^r \right)^{1/r} \right\|_{L^p(\nu)} \leq  \frac{c_{r,q_{j_1}} \cdots c_{r,q_{j_{m_2}}}}{c_{r,p}^{m_2}} \|T\| \prod_{i=1}^m \left\| \left( \sum_{k_i=1}^n |f_{k_i}^i|^r \right)^{1/r} \right\|_{L^{q_i}(\mu_i)},
$$
where $q_{j_1}, \dots, q_{j_{m_2}}$ are those $q_i\geq p$. Noting that $k_{q_i,p}(r) =1$ if $q_i<p$ and $k_{q_i,p}(r)=\frac{c_{r,q_i}}{c_{r,p}}$ if $q_i \geq p$ (see Theorem~\ref{constantes kqp para todos los triples}), the last inequality gives $k_{\q,p}(r) \leq k_{q_1,p}(r) \cdots k_{q_m,p}(r) < \infty$. The equality holds as a consequence of Proposition~\ref{relacion lineal-multilineal} and this proves (a).
\end{proof}
Note that in the proof of the statement (a) when $r<2$, the hypothesis $p, \qm < r$ is necessary in order to apply \eqref{ecuacion r estables}. When $r=2$, there is no need of this hypothesis to see, with the same proof, that $k_{\q,p}(2)< \infty$ (we will point out this fact in Proposition~\ref{MZ multilineal para p,q>0} below). In this case, we need the assumption $p, \qm \leq 2$ to ensure $k_{\q,p}(2)=k_{q_1,p}(2) \cdots k_{q_m,p}(2)$.

\subsection{Proof of the main theorem}\label{section main theorem}

We prove now Theorem~\ref{main theorem}, which determines the set of those triples $(p,\q,r)$ satisfying $k_{\q,p}(r)<\infty$, with the exception of the cases $p=\infty$ (partially discussed in Section~\ref{caso p infinito}) and $2 \leq p < \max\{q_1, \dots, q_m\}$. At this point, the results stated in the theorem are simple consequences of those obtained in the previous sections.

\begin{proof}[Proof of Theorem~\ref{main theorem}]
We begin by proving (i). Items (ia) and (ic) are consequences of Theorem~\ref{caracterizacion caso menor 2}, then we only need to prove (ib). If $k_{\q,p}(r)< \infty$ then $k_{\qm,p}(r)< \infty$ (by Proposition~\ref{relacion lineal-multilineal}) and by Theorem~\ref{constantes kqp para todos los triples}(ii) it follows that $2\leq r \leq p$. For the converse, take $2 \leq r \leq p$. If $\qm \leq r \leq p$ then $k_{\q,p}(r) = 1$ by Remark~\ref{remark after monotonicity}, while if $2 \leq r \leq \qm$ then $k_{\q,p}(r) < \infty$ by Corollary~\ref{interpolation property of r}, since $k_{\q,p}(2)<\infty$ (this is the particular case of the Marcinkiewicz-Zygmund inequalities proved in \cite{BomPerVil, GraMar}) and $k_{\q,p}(\qm) = 1 < \infty$.

Now we prove (ii). Item (iia) follows from Theorem~\ref{caracterizacion caso menor 2}. For (iib) (resp. (iic)) note that $k_{\q,p}(r)< \infty$ implies $k_{\qm,p}(r)< \infty$ and, by Theorem~\ref{constantes kqp para todos los triples}(iii), this gives $2 \leq r <p$ (resp. $r=2$). The converse in (iic) and the sufficient condition in (iib) are again consequences of the multilinear Marcinkiewicz-Zygmund inequality for $r=2$ obtained in \cite{BomPerVil, GraMar}.
\end{proof}

\section{Proof of Proposition~\ref{proposicion caso infinito}}\label{caso p infinito}
The key tool in this section is a variant of a Kahane-Salem-Zygmund inequality.
The idea is to construct a multilinear operator from $\ell^{q_1}_n \times \cdots \times \ell^{q_m}_n$ into $\ell^{\infty}_n$ with relatively small supremum norm but for which $\left( \sum_{i_1, \dots, i_m=1}^n |T(e_{i_1}, \dots, e_{i_m})|^r \right)^{1/r}$ has a relatively large $\ell^\infty$-norm. This, together with an appropriate choice of the indices $q_1, \dots, q_m, r$, will force $\lim_n k_{\q, \infty}^{(n)}(r) = \infty$ as desired.
In \cite[Thm. 4]{Boa}, Boas showed that there exists an $(m+1)$-linear map $T \colon \ell^{2}_n \times \cdots \times \ell^{2}_n \to \R$ of the form
\begin{equation}\label{T unimodular escalar}
T(z^{(1)}, \dots, z^{(m+1)})= \sum_{j_1, \dots, j_{m+1}=1}^n \varepsilon_{j_1, \dots, j_{m+1}} z_{j_1}^{(1)} \cdots z_{j_{m+1}}^{(m+1)}
\end{equation}
where $\varepsilon_{j_1, \dots, j_{m+1}}=\pm 1$, such that $\|T\|_{\ell^{2}_n \times \cdots \times \ell^{2}_n \to \R}\prec n^{1/2}$ (here, $\prec$ means that there exist a constant $C_m>0$, depending only on $m$, such that $\|T\|_{\ell^{2}_n \times \cdots \times \ell^{2}_n \to \R}\leq C_m n^{1/2}$). The restriction of this operator to $\ell^{1}_n \times \cdots \times \ell^{1}_n$ has norm one and hence, by an interpolation argument, if $1<q<2$ (note that $\frac{1}{q}= \frac{1 - \theta}{1} + \frac{\theta}{2}$ with $\theta=\frac{2}{q'}$) then $T \colon \ell^{q}_n \times \cdots \times \ell^{q}_n \to \R$ has norm $\|T\|_{ \ell^{q}_n \times \cdots \times \ell^{q}_n \to \R}\prec 1^{1-\theta} n^{\frac{\theta}{2}} = n^{\frac{1}{q'}}$. It is clear then that, if $1 \leq q_1, \dots, q_{m+1} \leq 2$ and $\qm=\max\{q_1, \dots, q_{m+1}\}$, $T\colon \ell^{q_1}_n \times \cdots \times \ell^{q_{m+1}}_n \to \R$ has norm $\|T\|_{ \ell^{q_1}_n \times \cdots \times \ell^{q_{m+1}}_n \to \R} \leq \|T\|_{ \ell^{\qm}_n \times \cdots \times \ell^{\qm}_n \to \R} \prec n^{\frac{1}{\qm'}}$. If, instead, $1\leq q_1, \dots, q_{m+1} \leq \infty$ with $q_i>2$ for at least one $1 \leq i \leq m+1$, we can take $T \colon \ell^{2}_n \times \cdots \times \ell^{2}_n \to \R$ as in \eqref{T unimodular escalar} and compose it with the identities $id^n_{q_i,2} \colon \ell^{q_i}_n \to \ell^2_n$.
\begin{eqnarray*}
\|T\circ (id^n_{q_1,2}, \dots, id^n_{q_{m+1},2})\|_{ \ell^{q_1}_n \times \cdots \times \ell^{q_{m+1}}_n \to \R} & \leq & \|T\|_{\ell^{2}_n \times \cdots \times \ell^{2}_n \to \R} \|id^n_{q_1,2}\| \cdots \|id^n_{q_{m+1},2}\| \\
&\prec & n^{\frac{1}{2}}\|id^n_{q_1,2}\| \cdots \|id^n_{q_{m+1},2}\|.
\end{eqnarray*}
Note that $id^n_{q_i,2}$ has norm one if $q_i \leq 2$ and norm $n^{\frac{1}{2}-\frac{1}{q_i}}$ if $q_i>2$.
%
%
As a consequence of the previous observations and the isometric correspondence between $m$-linear maps $ \ell^{q_1}_n \times \cdots \times \ell^{q_{m}}_n \to \ell^{q_{m+1}'}_n$ and $(m+1)$-linear maps $ \ell^{q_1}_n \times \cdots \times \ell^{q_{m+1}}_n \to \R$, we obtain the following lemma.

\begin{lemma}\label{KSZ ineq}
Let $m,n \in \N$, $1 \leq q_1, \dots, q_m, p \leq \infty$ and $\qm= \max\{q_1, \dots, q_m\}$.
Then there exists an $m$-linear map $T \colon \ell^{q_1}_n \times \cdots \times \ell^{q_m}_n \to \ell^{p}_n$ of the form
\begin{equation}\label{T unimodular}
T(z^{(1)}, \dots, z^{(m)})= \sum_{j_{m+1}=1}^n\sum_{j_1, \dots, j_m=1}^n \varepsilon_{j_1, \dots, j_{m+1}} z_{j_1}^{(1)} \cdots z_{j_m}^{(m)} e_{j_{m+1}}
\end{equation}
where $\varepsilon_{j_1, \dots, j_{m+1}}=\pm 1$, such that
$$
\|T\| \prec
\left\{
\begin{array} {c l}
n^{\frac{1}{\min\{\qm', p\}}} & \text{if $\qm, p' \leq 2$,}\\
n^{\frac{1}{2}}\|id^n_{q_1,2}\| \cdots \|id^n_{q_{m},2}\| \|id^n_{p',2}\|  & \text{otherwise.}
\end{array}
\right.
$$
\end{lemma}

\begin{proof}[Proof of Proposition~\ref{proposicion caso infinito}]
Our goal is to show that $k_{\q, \infty}^{(n)}(r) \succ n^{\gamma}$ for some $\gamma>0$ independent of $n$,  in which case $k_{\q, \infty}(r)=\lim_n k_{\q, \infty}^{(n)}(r) = \infty$.  In virtue of Lemma~\ref{KSZ ineq}, there is an $m$-linear operator $T\colon \ell^{q_1}_{n} \times \cdots \times \ell^{q_m}_n \to \ell^\infty_n$ as in \eqref{T unimodular}
such that
$$
\|T\| \prec n^{ \frac{1}{\max(\qm',2)} +\sum_{q_i>2} \left( \frac{1}{2}-\frac{1}{q_i} \right)},
$$
where $\sum_{q_i>2} \left( \frac{1}{2}-\frac{1}{q_i} \right) =0$ if $\qm \leq 2$.
Now, since $|T(e_{i_1}, \dots, e_{i_m})| = (1, \dots, 1)$ for all $1 \leq i_1, \dots, i_m \leq n$ we have
$$
\left\| \left( \sum_{i_1, \dots, i_m=1}^n |T(e_{i_1}, \dots, e_{i_m})|^r \right)^{1/r} \right\|_{\ell^\infty} = \|((n^m, \dots, n^m))^{1/r}\|_{\ell^\infty} = n^{\frac{m}{r}}.
$$
Then,
\begin{eqnarray}
\nonumber n^{\frac{m}{r}}&=& \left\| \left( \sum_{i_1, \dots, i_m=1}^n |T(e_{i_1}, \dots, e_{i_m})|^r \right)^{1/r} \right\|_{\ell^\infty}
\leq k_{\q, \infty}^{(n)}(r) \|T\| \prod_{i=1}^m \left\|\left(\sum_{j=1}^n |e_j|^r \right)^{1/r} \right\|_{\ell^{q_i}}\\
&\prec& k_{\q, \infty}^{(n)}(r) n^{ \frac{1}{\max(\qm',2)} + \frac{1}{\min(q_1,2)}+\cdots +\frac{1}{\min(q_m,2)}} \label{acotacion q por 1}.
\end{eqnarray}
%
The statement follows  from this inequality.
\end{proof}

\section{Applications}\label{applications}

\subsection{Weighted vector-valued estimates for multilinear Calder\'on-Zygmund operators}\label{vector-valued CZ operators}

The study of multilinear Calder\'on-Zygmund theory finds its origins in the seventies with works of Coifman and Meyer, but a systematic treatment of this topic appears later with works of Grafakos and Torres \cite{GraTor1, GraTor2}. Recall the definition of a multilinear Calder\'on-Zygmund operator:
let $T:S({\mathbb R}^n)\times\dots\times S({\mathbb R}^n)\to
S'({\mathbb R}^n)$ be a multilinear operator initially defined on the $m$-fold
product of Schwartz spaces and taking values into the space of
tempered distributions;
we say that $T$ is an $m$-linear
Calder\'on-Zygmund operator if, for some $1\le q_1, \dots, q_m<\infty$ and $\frac{1}{m} \leq p < \infty$ satisfying $\frac{1}{p}=\frac{1}{q_1}+\dots+\frac{1}{q_m}$, it
extends to a bounded multilinear operator from
$L^{q_1}\times\dots\times L^{q_m}$ to $L^p$, and if
there exists a function
$K$ defined
off the diagonal $x=y_1=\dots=y_m$ in $({\mathbb
R}^n)^{m+1}$ satisfying the appropriate decay and smoothness conditions (see \cite{GraTor1, GraTor2}) and such that
$$T(f^1,\dots,f^m)(x) =\int_{\R^n} \cdots \int_{\R^n}
K(x,y_1,\dots,y_m) \prod_{i=1}^m f^i(y_i)\, dy_1\cdots dy_m
$$
for all $x \notin \cap_{i=1}^{m} \supp f_i$.
In \cite{GraTor1} it was shown that, if $\frac{1}{p}=\frac{1}{q_1}+\dots+\frac{1}{q_m}$ with $1 < q_1, \dots, q_m < \infty$, then an $m$-linear Calder\'on-Zygmund operator $T$ maps $L^{q_1}(\R^n)\times\dots\times L^{q_m}(\R^n)$ into $L^p(\R^n)$. If $1 \leq q_1, \dots, q_m < \infty$ and at least one $q_i=1$, then $T$ maps $L^{q_1}(\R^n)\times\dots\times L^{q_m}(\R^n)$ into $L^{p, \infty}(\R^n)$. Regarding the weighted norm inequalities, the first result was obtained in \cite{GraTor2} (see also \cite{PerTor}) where the authors proved that, if $1 < q_1, \dots, q_m < \infty$ and $w$ is a weight in the Muckenhoupt $A_{q_0}$ class for $q_0=\min\{q_1, \dots, q_m\}$, an $m$-linear Calder\'on-Zygmund operator $T$ maps $L^{q_1}(w)\times\dots\times L^{q_m}(w)$ into $L^p(w)$.
 The same approach of \cite{GraTor2} shows that $T$ maps
$L^{q_1}(w_1)\times \dots \times L^{q_m}(w_m) $ into $ L^{p}(\nu_{\vec w})$,
where $\nu_{\vec w}=\prod_{i=1}^m w_i^{p/q_i}$ and $w_i$ is in $A_{q_i}$.
As expected, if at least one $q_i=1$, then the weak endpoint estimate holds.
In \cite{LOPTT}, the authors developed the right class of multiple weights
for $m$-linear Calder\'on-Zygmund operators. We briefly review the weighted estimate proved in there for multilinear Calder\'on-Zygmund operators.
As usual, let  $1\le q_1, \dots , q_{m}<\infty$ and $\frac{1}{m} \leq p < \infty$ be such that $\frac{1}{p}=\frac{1}{q_1}+\dots+\frac{1}{q_m}$.
We say that $\vec w=(w_1,\dots,w_m)$ satisfies the \emph{multilinear $A_{\q}$ condition} if
\begin{equation}\label{multiap}
\sup_{Q}\Big(\frac{1}{|Q|}\int_Q\nu_{\vec
w}\Big)^{1/p}\prod_{i=1}^m\Big(\frac{1}{|Q|}\int_Q
w_i^{1-q'_i}\Big)^{1/q'_i}<\infty
\end{equation}
where the supremum is taken over all cubes $Q$
(when $q_i=1$, $\Big(\frac{1}{|Q|}\int_Q w_i^{1-q'_i}\Big)^{1/q'_i}$
is understood as $\displaystyle(\inf_Q w_i)^{-1}$). Now, if $\vec w$ satisfies the $A_{\q}$ condition and $1<q_1,\dots,q_m<\infty$, then an $m$-linear Calder\'on-Zygmund operator $T$ maps $L^{q_1}(w_1)\times \dots \times L^{q_m}(w_m) $ into $ L^{p}(\nu_{\vec w})$. If at least one $q_i=1$, then $T$ maps $L^{q_1}(w_1)\times \dots \times L^{q_m}(w_m) $ into $ L^{p,\infty}(\nu_{\vec w})$. It is shown that $\prod_{i=1}^m A_{q_i} \subseteq A_{\q}$ and that this inclusion is strict; then, the above weighted estimates improve those obtained in \cite{GraTor2}. In fact, if $T$ is the $m$-linear Riesz Transform, it was proved in \cite{LOPTT} that $A_{\q}$ is a necessary condition for such weighted estimate of $T$.

\subsubsection{Some known estimates and their comparison with Marcinkiewicz-Zygmund inequalities}\label{comparison}

In the context of vector-valued inequalities for multilinear  Calder\'on-Zygmund operators, the following result was proved in \cite[Corollary 3.3]{CruMarPer} by means of extrapolation techniques. See also \cite[Section 6.4]{CurGarMarPer} where this type of inequalities are obtained in the more general context of rearrangement invariant quasi-Banach function spaces.

\begin{theorem}[\cite{CruMarPer}]\label{teorema CruMarPer}
Let $T$ be a multilinear Calder\'on-Zygmund operator, $1 \leq q_1, \dots, q_m < \infty$, $1 < r_1, \dots, r_m < \infty$ and $0<p,r < \infty$ such that
$$
\frac{1}{p}=\frac{1}{q_1}+\cdots+\frac{1}{q_m}, \quad  \frac{1}{r}=\frac{1}{r_1}+\cdots+\frac{1}{r_m}.
$$
If $1<q_1, \dots, q_m < \infty$ and $w \in A_{q_0}$ for $q_0=\min\{q_1, \dots, q_m\}$, then
\begin{equation}\label{strong (p,p)}
\left\| \left( \sum_k |T(f^1_k, \dots, f^m_k)|^r \right)^{\frac{1}{r}} \right\|_{L^p(w)} \leq C \prod_{i=1}^m \left\| \left( \sum_k |f^i_k|^{r_i} \right)^{\frac{1}{r_i}} \right\|_{L^{q_i}(w)}.
\end{equation}
If at least one $q_i=1$ and $w \in A_1$, then
\begin{equation}\label{weak (p,p)}
\left\| \left( \sum_k |T(f^1_k, \dots, f^m_k)|^r \right)^{\frac{1}{r}} \right\|_{L^{p,\infty}(w)} \leq C \prod_{i=1}^m \left\| \left( \sum_k |f^i_k|^{r_i} \right)^{\frac{1}{r_i}} \right\|_{L^{q_i}(w)}.
\end{equation}
\end{theorem}
The estimate \eqref{strong (p,p)} was obtained independently in \cite{GraMar} as a particular case of the inequality \eqref{strong (p,p) m weights} for $m$-tuples of weights. Also a weaker version of \eqref{weak (p,p)} was obtained as a consequence of a multilinear Marcinkiewicz-Zygmund inequality for $r=2$.

\begin{theorem}[\cite{GraMar}]\label{teorema GraMar}
Let $T$ be a multilinear Calder\'on-Zygmund operator, $1 \leq q_1, \dots, q_m < \infty$, $1 < r_1, \dots, r_m < \infty$ and $0<p,r < \infty$ such that
$$
\frac{1}{p}=\frac{1}{q_1}+\cdots+\frac{1}{q_m}, \quad  \frac{1}{r}=\frac{1}{r_1}+\cdots+\frac{1}{r_m}.
$$
If $1 < q_1, \dots, q_m < \infty$ and $(w_1^{q_1}, \dots, w_{m}^{q_m}) \in  A_{q_1} \times \cdots \times A_{q_m}$, then
\begin{equation}\label{strong (p,p) m weights}
\left\| \left( \sum_k |T(f^1_k, \dots, f^m_k)|^r \right)^{\frac{1}{r}} \right\|_{L^p(w_1^p \cdots w_m^p)} \leq C \prod_{i=1}^m \left\| \left( \sum_k |f^i_k|^{r_i} \right)^{\frac{1}{r_i}} \right\|_{L^{q_i}(w_i^{q_i})}.
\end{equation}
If at least one $q_i=1$ and $w \in A_1$, then
\begin{equation}\label{weak (p,p) GraMar}
\left\| \left( \sum_k |T(f^1_k, \dots, f^m_k)|^2 \right)^{\frac{1}{2}} \right\|_{L^{p,\infty}(w)} \leq C \prod_{i=1}^m \left\| \left( \sum_k |f^i_k|^{2} \right)^{\frac{1}{2}} \right\|_{L^{q_i}(w)}.
\end{equation}
\end{theorem}
It is worth mentioning that, in the above weighted vector-valued estimates, the multilinear $A_{\q}$ condition is not considered. We will consider this appropriate class of multiple weights in Corollary~\ref{vector-valued for CZ}, where vector-valued inequalities for Calder\'on-Zygmund operators are obtained.
Also, as we emphasize below, the vector-valued inequalities that we obtain are significantly different from those stated in \eqref{strong (p,p)}, \eqref{weak (p,p)}, \eqref{strong (p,p) m weights}. In our case, the sum on the left-hand side is replaced by a multi-indexed sum over $k_1, \dots, k_m$ and we consider only one power $r$ instead of the powers $1<r_1, \dots, r_m < \infty$.

At this point we would like to stress the relation $ \frac{1}{r}=\frac{1}{r_1}+\cdots+\frac{1}{r_m}$ that appears in the hypotheses of the previous theorems which, by the way, shows why the estimate \eqref{weak (p,p) GraMar} is weaker than \eqref{weak (p,p)}. When proving estimates like \eqref{strong (p,p) m weights}, one is interested in the study of the following inequalities:
\begin{equation}\label{MZ multilineal para un k}
\left\| \left( \sum_k |T(f^1_k, \dots, f^m_k)|^s \right)^{\frac{1}{s}} \right\|_{L^p(\nu)} \leq C \prod_{i=1}^m \left\| \left( \sum_k |f^i_k|^{r_i} \right)^{\frac{1}{r_i}} \right\|_{L^{q_i}(\mu_i)}.
\end{equation}
If in the previous inequality we put $f^i_k=f^i \in L^{q_i}(\mu_i)$ for $1 \leq k \leq n$, $1\leq i \leq m$ and $f^i_k=0$ otherwise, we obtain
\begin{equation}\label{optimalidad para un k}
n^{\frac{1}{s}} \|T(f^1, \dots, f^m)\|_p \leq C \prod_{i=1}^m n^{\frac{1}{r_i}} \|f^i\|_{L^{q_i}(\mu_i)}
\end{equation}
and, since $C$ is independent of $n$, this forces
$
\frac{1}{s} \leq \frac{1}{r_1}+\cdots+\frac{1}{r_m}.
$
Thus, the estimate \eqref{MZ multilineal para un k} is optimal when $s=r$, where $r$ satisfies $\frac{1}{r} = \frac{1}{r_1}+\cdots+\frac{1}{r_m}$. In the case of \eqref{weak (p,p) GraMar} the relation between the powers is not optimal as in \eqref{weak (p,p)}. In fact, if \eqref{weak (p,p)} holds for $r=2$ then $2 \leq r_i$ for all $1 \leq i \leq m$ and, since $\ell^2 \hookrightarrow \ell^{r_i}$ with $\|\cdot\|_{\ell^{r_i}} \leq \|\cdot\|_{\ell^2}$, then
$$
 \left\| \left( \sum_k |f^i_k|^{r_i} \right)^{\frac{1}{r_i}} \right\|_{L^{q_i}(\mu_i)} \leq  \left\| \left( \sum_k |f^i_k|^{2} \right)^{\frac{1}{2}} \right\|_{L^{q_i}(\mu_i)},
$$
which yields the (hence, weaker) estimate in \eqref{weak (p,p) GraMar}.

Now, the Marcinkiewicz-Zygmund inequalities we were studying in the previous sections have a significant difference with those in \eqref{MZ multilineal para un k}. In the former ones the sum
$$\left( \sum_{k_1, \dots, k_m} |T(f^1_{k_1}, \dots, f^m_{k_m})|^r \right)^{\frac{1}{r}}$$ ranges over the indices $k_1, \dots, k_m$, while in \eqref{MZ multilineal para un k} we sum over the diagonal, hence only one index $k$. This produces another optimal relation between the powers $r, r_1, \dots, r_m$.
 Indeed, suppose we are interested in estimates of the form
\begin{equation}\label{optimal powers MZ multilineal}
\left\| \left( \sum_{k_1, \dots, k_m} |T(f^1_{k_1}, \dots, f^m_{k_m})|^s \right)^{\frac{1}{s}}  \right\|_{L^p(\nu)} \leq C \prod_{i=1}^m \left\| \left( \sum_{k_i} |f^i_{k_i}|^{r_i} \right)^{\frac{1}{r_i}} \right\|_{L^{q_i}(\mu_i)}.
\end{equation}
If we choose $f^i_{k_i} = f^i \in L^{q_i}(\mu_i)$ for $1 \leq k_i \leq n$, $1 \leq i \leq m$ and $f^i_{k_i}=0$ otherwise, we have
$$
n^{\frac{m}{s}} \|T(f^1, \dots, f^m)\|_p \leq C \prod_{i=1}^m n^{\frac{1}{r_i}} \|f^i\|_{L^{q_i}(\mu_i)},
$$
from where we get $\frac{m}{s} \leq \frac{1}{r_1}+\cdots+\frac{1}{r_m}$. Also, if in \eqref{optimal powers MZ multilineal} we take any sequence $\{f^1_{k_1}\}_{k_1} \subset L^{q_1}(\mu_1)$ and, for each $2 \leq i \leq m$, we put $f^i_1=f^i \in L^{q_i}(\mu_i)$ and $f^i_{k_i}=0$ for $k_i\geq 2$, then we obtain
$$
\left\| \left( \sum_{k_1} |\tilde{T}(f^1_{k_1})|^s \right)^{\frac{1}{s}}  \right\|_{L^p(\nu)} \leq C \left(  \prod_{i=2}^m \|f^i\|_{L^{q_i}(\mu_i)} \right) \left\| \left( \sum_{k_1} |f^1_{k_1}|^{r_1} \right)^{\frac{1}{r_1}} \right\|_{L^{q_1}(\mu_1)}
$$
where $\tilde{T}\colon L^{q_1}(\mu_1) \to L^p(\nu)$  is the linear operator $\tilde{T}(\cdot) = T(\cdot, f^2, \dots, f^m)$. In virtue of \eqref{optimalidad para un k} (when $m=1$) this last inequality gives $\frac{1}{s}\leq \frac{1}{r_1}$. Analogously, $\frac{1}{s}\leq \frac{1}{r_i}$ for all $1 \leq i \leq m$. In sum, when we are dealing with inequality \eqref{optimal powers MZ multilineal} we have the following conditions over the powers $s, r_1, \dots, r_m$,
$$
\frac{m}{s} \leq \frac{1}{r_1}+\cdots+\frac{1}{r_m} \quad \text{and} \quad \frac{1}{s}\leq \frac{1}{r_i} \quad \text{for all $1 \leq i \leq m$.}
$$
This shows that the inequality  \eqref{optimal powers MZ multilineal} is optimal when $s=r_1=\cdots = r_m=r$, which is just the case treated in the Marcinkiewicz-Zygmund inequalities.

Camil Muscalu pointed out to us that, in the unweighted case, Marcinkiewicz-Zygmund estimates for multilinear Calder\'on-Zygmund operators can be deduced from more general multiple vector-valued inequalities obtained in his recent works with Cristina Benea. In \cite{BenMus, BenMus2}, the authors developed a powerful method, which they called the helicoidal method, that allows to obtain vector-valued inequalities in harmonic analysis. They apply this method to obtain vector-valued inequalities for paraproducts (which can be regarded as bilinear multiplier operators) and the bilinear Hilbert transform. We also refer to \cite{CulDiPOu, DiPOu} where, with different methods, the authors obtain vector-valued inequalities for multilinear multiplier operators.
As far as we know, Marcinkiewicz-Zygmund inequalities for multilinear Calder\'on-Zygmund operators were not addressed in the weighted case. Also, although it might be possible to derive similar estimates via the helicoidal method, our approach is completely different.

\subsubsection{Weighted vector-valued estimates}
We state now as a proposition a result that is partially demonstrated in the proof of Theorem~\ref{caracterizacion caso menor 2}.

\begin{proposition}\label{MZ multilineal para p,q>0}
Let $0<p, q_1, \dots, q_m < r<2$ or $r=2$ and $0<p, q_1, \dots, q_m < \infty$ and, for each $1 \leq i \leq m$, consider $\{f^i_{k_i}\}_{k_i} \subset L^{q_i}(\mu_i)$. Then, there exists a constant $C>0$ such that the following estimates hold.
\begin{enumerate}
\item[\rm (i)] If $T\colon L^{q_1}(\mu_1) \times \cdots \times L^{q_m}(\mu_m) \to L^p(\nu)$ then
\begin{equation}\label{strong MZ multilineal p,q>0}
\left\| \left( \sum_{k_1, \dots, k_m} |T(f^1_{k_1}, \dots, f^m_{k_m})|^r \right)^{\frac{1}{r}}  \right\|_{L^p(\nu)} \leq C \|T\| \prod_{i=1}^m \left\| \left( \sum_{k_i} |f^i_{k_i}|^{r} \right)^{\frac{1}{r}} \right\|_{L^{q_i}(\mu_i)}.
\end{equation}
\item[\rm (ii)] If $T\colon L^{q_1}(\mu_1) \times \cdots \times L^{q_m}(\mu_m) \to L^{p,\infty}(\nu)$ then
\begin{equation}\label{weak MZ multilineal p,q>0}
\left\| \left( \sum_{k_1, \dots, k_m} |T(f^1_{k_1}, \dots, f^m_{k_m})|^r \right)^{\frac{1}{r}}  \right\|_{L^{p, \infty}(\nu)} \leq C e^{\frac{1}{p}} \|T\|_{weak} \prod_{i=1}^m \left\| \left( \sum_{k_i} |f^i_{k_i}|^{r} \right)^{\frac{1}{r}} \right\|_{L^{q_i}(\mu_i)}.
\end{equation}
\end{enumerate}
\end{proposition}

\begin{proof}
The inequality \eqref{strong MZ multilineal p,q>0} was proved in Theorem~\ref{caracterizacion caso menor 2} for $1 \leq p, q_1, \dots, q_m < r<2$ (see the assertion (a) inside the proof). It is easy to check that, with exactly the same proof, the statement remains valid for $0<p, q_1, \dots, q_m < r<2$. The case $r=2$ also follows the same proof, with slight modifications in the involved constants. See the case $r=2<p$ in  Lemma~\ref{lemma r-stable} (see also \cite[Thm. 6 (a)]{GraMar}).
Then, we only need to prove (ii) and, for this purpose, we follow ideas of \cite[Thm. 6 (b)]{GraMar} where the particular case $r=2$ is addressed. Recall that, for $0<s<p<\infty$,
\begin{equation}\label{caracterizacion weak Lp}
\|f\|_{L^{p,\infty}(\nu)} \leq \sup_{0<\nu(E)<\infty} \nu(E)^{\frac{1}{p}-\frac{1}{s}} \left(\int_E |f|^s d\nu\right)^{1/s} \leq \left( \frac{p}{p-s}\right)^{\frac{1}{s}} \|f\|_{L^{p,\infty}(\nu)}.
\end{equation}
Then,
\begin{eqnarray}\label{acotacion por T_E}
\nonumber && \left\| \left( \sum_{k_1, \dots, k_m} |T(f^1_{k_1}, \dots, f^m_{k_m})|^r \right)^{\frac{1}{r}}  \right\|_{L^{p, \infty}(\nu)} \leq\\
\nonumber &\leq& \sup_{0<\nu(E)<\infty} \nu(E)^{\frac{1}{p}-\frac{1}{s}} \left( \int_E \left( \sum_{k_1, \dots, k_m} |T(f^1_{k_1}, \dots, f^m_{k_m})(\omega)|^r \right)^{\frac{s}{r}} d\nu(\omega) \right)^{1/s}\\
&=& \sup_{0<\nu(E)<\infty} \nu(E)^{\frac{1}{p}-\frac{1}{s}} \left( \int \left( \sum_{k_1, \dots, k_m} |\chi_E(\omega) T(f^1_{k_1}, \dots, f^m_{k_m})(\omega)|^r \right)^{\frac{s}{r}} d\nu(\omega) \right)^{1/s}.
\end{eqnarray}
Now for each measurable set $E$ of positive and finite $\nu$-measure, consider $T_E\colon L^{q_1}(\mu_1) \times \cdots \times L^{q_m}(\mu_m) \to L^s(\nu)$ defined by
$$
T_E(f^1, \dots, f^m)(\omega)= \chi_E(\omega) T(f^1, \dots, f^m)(\omega)
$$
(the fact that $T_E$ takes values in $L^s(\nu)$ follows from \eqref{caracterizacion weak Lp}). In virtue of the second inequality in \eqref{caracterizacion weak Lp} we have
\begin{eqnarray*}
 \nu(E)^{\frac{1}{p}-\frac{1}{s}} \|T_E(f^1, \dots, f^m)\|_{L^s(\nu)} &\leq& \left( \frac{p}{p-s}\right)^{\frac{1}{s}} \|T(f^1, \dots, f^m)\|_{L^{p, \infty}(\nu)}\\
 &\leq&  \left( \frac{p}{p-s}\right)^{\frac{1}{s}} \|T\|_{weak} \|f^1\|_{L^{q_1}(\mu_1)} \cdots \|f^m\|_{L^{q_m}(\mu_m)}
\end{eqnarray*}
and, consequently,  $\nu(E)^{\frac{1}{p}-\frac{1}{s}} \|T_E\| \leq  \left( \frac{p}{p-s}\right)^{\frac{1}{s}} \|T\|_{weak}$. Going back to \eqref{acotacion por T_E}, we obtain
\begin{eqnarray*}
&& \left\| \left( \sum_{k_1, \dots, k_m} |T(f_1^{k_1}, \dots, f_m^{k_m})|^r \right)^{\frac{1}{r}}  \right\|_{L^{p, \infty}(\nu)} \leq\\
&\leq& \sup_{0<\nu(E)<\infty} \nu(E)^{\frac{1}{p}-\frac{1}{s}} \left\| \left( \sum_{k_1, \dots, k_m} |T_E(f^1_{k_1}, \dots, f^m_{k_m})|^r \right)^{\frac{1}{r}} \right\|_s\\
&\leq&  \sup_{0<\nu(E)<\infty} \nu(E)^{\frac{1}{p}-\frac{1}{s}} C \|T_E\| \prod_{i=1}^m \left\| \left( \sum_{k_i} |f^i_{k_i}|^{r} \right)^{\frac{1}{r}} \right\|_{L^{q_i}(\mu_i)} \quad \text{(by item (i))}\\
&\leq& C  \left( \frac{p}{p-s}\right)^{\frac{1}{s}} \|T\|_{weak}  \prod_{i=1}^m \left\| \left( \sum_{k_i} |f^i_{k_i}|^{r} \right)^{\frac{1}{r}} \right\|_{L^{q_i}(\mu_i)}
\end{eqnarray*}
and since $0<s<p$ was arbitrary, letting $s \to 0$ we obtain the desired estimate.
\end{proof}

As a consequence of the previous proposition, we obtain the following vector-valued estimates for multilinear Calder\'on-Zygmund operators which should be compared with those of Theorems~\ref{teorema CruMarPer} and \ref{teorema GraMar}.

\begin{corollary}\label{vector-valued for CZ}
Let $T$ be a multilinear Calder\'on-Zygmund operator, $1 \leq q_1, \dots, q_m <r < 2$ (or $r=2$ and $1 \leq q_1, \dots, q_m <\infty$) and $p>0$ such that
$
\frac{1}{p}=\frac{1}{q_1}+\cdots+\frac{1}{q_m}.
$
Suppose $\vec w=(w_1,\dots,w_m)$ satisfies the multilinear $A_{\q}$ condition and consider $\nu_{\vec w}=\prod_{i=1}^m w_i^{p/q_i}$. Then, there exists a constant $C>0$ such that the following estimates hold.
\begin{enumerate}
\item[\rm (i)] If $q_i>1$ for all $1 \leq i \leq m$, then
\begin{equation}
\left\|  \left( \sum_{k_1, \dots, k_m} |T(f^1_{k_1}, \dots, f^m_{k_m})|^r \right)^{\frac{1}{r}} \right\|_{L^p(\nu_{\vec w})} \leq C \|T\| \prod_{i=1}^m \left\| \left( \sum_{k_i} |f^i_{k_i}|^{r} \right)^{\frac{1}{r}} \right\|_{L^{q_i}(w_i)}.
\end{equation}
\item[\rm (ii)] If at least one $q_i=1$, then
\begin{equation}
\left\|  \left( \sum_{k_1, \dots, k_m} |T(f^1_{k_1}, \dots, f^m_{k_m})|^r \right)^{\frac{1}{r}}  \right\|_{L^{p,\infty}(\nu_{\vec w})} \leq C e^{\frac{1}{p}} \|T\|_{weak} \prod_{i=1}^m \left\| \left( \sum_{k_i} |f^i_{k_i}|^{r} \right)^{\frac{1}{r}} \right\|_{L^{q_i}(w_i)}.
\end{equation}
\end{enumerate}
\end{corollary}

\subsection{Marcinkiewicz-Zygmund inequalities for positive multilinear operators}\label{section MZ positive}
Following the proof of the Marcinkiewicz-Zygmund inequality for positive linear operators stated in \cite[Chapter V.1, Thm. 1.12]{GarRub}, we prove Proposition~\ref{MZ positive} which extends this result to the multilinear setting. Recall that a multilinear operator $T\colon L^{q_1}(\mu_1)\times \cdots L^{q_m}(\mu_m) \to L^p(\nu)$ is positive if $f^1, \dots, f^m \geq 0$ implies $T(f^1, \dots, f^m) \geq 0$. It can be seen that if $T$ is positive, then $|T(f^1, \dots, f^m)| \leq T(g^1, \dots, g^m)$ whenever $|f^1|\leq g^1, \dots, |f^m|\leq g^m$.

\begin{proof}[Proof of Proposition~\ref{MZ positive}]
It suffices to show that
\begin{equation}\label{suffice MZ positive}
\left(\sum_{k_1, \dots, k_m} |T(f_{k_1}^1, \dots, f_{k_m}^m)(\omega)|^r\right)^{1/r} \leq T\left( \left(\sum_{k_1=1}^{n_1} |f_{k_1}^1|^r\right)^{1/r}, \dots, \left(\sum_{k_m=1}^{n_m} |f_{k_m}^m|^r\right)^{1/r} \right)(\omega)
\end{equation}
for any choice of functions $\{f_{k_i}^i\}_{k_i=1}^{n_i} \subset L^{q_i}(\mu_i)$, $1\leq i \leq m$.
The case $r=\infty$ is immediate by the positivity of $T$. Then, we assume $1 \leq r < \infty$ and prove \eqref{suffice MZ positive} by induction on $m$. The case $m=1$ is proved in \cite[Chapter V.1, Thm. 1.12]{GarRub} via a duality argument. Now, let $m\geq 2$ and suppose \eqref{suffice MZ positive} holds for $m-1$.
By duality we know that, given $\mathbf{a}=(a_{k_1 \dots k_m})_{k_1, \dots, k_m=1}^\infty$,
$$
\|\mathbf{a}\|_ {\ell^r(\N \times \cdots \times \N)} = \sup \left| \sum_{k_1, \dots, k_m} a_{k_1 \dots k_m} b_{k_1 \dots k_m} \right|
$$
where the supremum is taken over all $\mathbf{b} = (b_{k_1 \dots k_m})_{k_1, \dots, k_m}$ such that $\|\mathbf{b}\|_{\ell^{r'}(\N \times \cdots \times \N)}\leq 1$. Take any $\mathbf{b} \in \ell^{r'}(\N \times \cdots \times \N)$ and note that
\begin{eqnarray*}
&& \left| \sum_{k_1, \dots, k_m} T(f_{k_1}^1, \dots, f_{k_m}^m)(\omega) b_{k_1 \dots k_m} \right| = \left| \sum_{k_1, \dots, k_{m-1}} T\left(f_{k_1}^1, \dots, f_{k_{m-1}}^{m-1}, \sum_{k_m} b_{k_1 \dots k_m} f_{k_m}^m\right)(\omega) \right| \\
&\leq& \sum_{k_1, \dots, k_{m-1}} \left|T\left(f_{k_1}^1, \dots, f_{k_{m-1}}^{m-1}, \sum_{k_m} b_{k_1 \dots k_m} f_{k_m}^m\right)(\omega)\right| \\
&\leq& \sum_{k_1, \dots, k_{m-1}} T\left(|f_{k_1}^1|, \dots, |f_{k_{m-1}}^{m-1}|, \left( \sum_{k_m} |f_ {k_m}^m|^r \right)^{1/r} \right)(\omega)  \left( \sum_{k_m} |b_{k_1 \dots k_m}|^{r'} \right)^{1/r'},
\end{eqnarray*}
where the last inequality follows from  H$\ddot{\text{o}}$lder's inequality and the positivity of $T$. A repeated application of H$\ddot{\text{o}}$lder's inequality and the induction hypothesis shows that
\begin{eqnarray*}
&&\sum_{k_1, \dots, k_{m-1}} T\left(|f_{k_1}^1|, \dots, |f_{k_{m-1}}^{m-1}|, \left( \sum_{k_m} |f_ {k_m}^m|^r \right)^{1/r} \right)(\omega)  \left( \sum_{k_m} |b_{k_1 \dots k_m}|^{r'} \right)^{1/r'} \\
&\leq&
\left(\sum_{k_1, \dots, k_{m-1}}  \left| T\left(|f_{k_1}^1|, \dots, |f_{k_{m-1}}^{m-1}|, \left( \sum_{k_m} |f_ {k_m}^m|^r \right)^{1/r} \right)(\omega)  \right|^r\right)^{1/r} \|\mathbf{b}\|_{\ell^{r'}(\N \times \cdots \times \N)} \\
&\leq&  T\left( \left(\sum_{k_1=1}^{n_1} |f_{k_1}^1|^r\right)^{1/r}, \dots, \left(\sum_{k_m=1}^{n_m} |f_{k_m}^m|^r\right)^{1/r} \right)(\omega) \|\mathbf{b}\|_{\ell^{r'}(\N \times \cdots \times \N)}
\end{eqnarray*}
Hence,
$$
\left| \sum_{k_1, \dots, k_m} T(f_{k_1}^1, \dots, f_{k_m}^m)(\omega) b_{k_1 \dots k_m} \right| \leq
T\left(\left( \sum_{k_{1}} |f_ {k_{1}}^{1}|^r \right)^{1/r}, \dots, \left( \sum_{k_m} |f_ {k_m}^m|^r \right)^{1/r} \right)(\omega) \|\mathbf{b}\|_{\ell^{r'}(\N \times \cdots \times \N)},
$$
and taking supremum over all $\|\mathbf{b}\|_{\ell^{r'}(\N \times \cdots \times \N)} \leq 1$ we get \eqref{suffice MZ positive}.
\end{proof}

As an immediate consequence of Proposition~\ref{MZ positive} and Young's inequality, we have the following vector-valued inequality for the convolution $f*g(x) = \int_\R f(x-y) g(y) \, dy$.
\begin{corollary}
Let $1 \leq q_1,q_2,p \leq \infty$ satisfying $\frac{1}{q_1} + \frac{1}{q_2} = \frac{1}{p} + 1$ and let $1 \leq r \leq \infty$. Then,
$$
\left\|\left(\sum_{k_1, k_2} |f_{k_1}^1 * f_{k_2}^2|^r\right)^{1/r}\right\|_{L^p(\R)} \leq  \left\| \left(\sum_{k_1=1}^{n_1} |f_{k_1}^1|^r\right)^{1/r} \right\|_{L^{q_1}(\R)} \left\| \left(\sum_{k_2=1}^{n_2} |f_{k_2}^2|^r\right)^{1/r} \right\|_{L^{q_2}(\R)}
$$
for any choice of functions $\{f_{k_i}^i\}_{k_i=1}^{n_i} \subset L^{q_i}(\R)$, $1\leq i \leq 2$.
\end{corollary}

The inequality \eqref{ineq MZ positive} should be compared with \cite[Thm. 6.2]{DefMas}, where it is proved that, given a positive $m$-linear operator $T\colon X_1 \times \cdots \times X_m \to Y$ between Banach lattices and $1 \leq r_1, \dots, r_m, r \leq \infty$ such that $\frac{1}{r}=\frac{1}{r_1} + \cdots + \frac{1}{r_m}$, we have
$$
\left\|\left(\sum_{k} |T(x_{k}^1, \dots, x_{k}^m)|^r\right)^{1/r}\right\|_{Y} \leq  \|T\| \prod_{i=1}^m \left\| \left(\sum_{k=1}^{n} |x_{k}^i|^{r_i}\right)^{1/r_i} \right\|_{X_i}
$$
for any choice of sequences $\{x_{k}^i\}_{k=1}^{n} \subset X_i$, $1\leq i \leq m$. As a consequence, if $1 \leq q_1,q_2,p \leq \infty$ satisfy $\frac{1}{q_1} + \frac{1}{q_2} = \frac{1}{p} + 1$ and $1 \leq r_1, r_2, r \leq \infty$ are such that $\frac{1}{r}=\frac{1}{r_1} + \frac{1}{r_2}$, then
$$
\left\|\left(\sum_{k=1}^n |f_{k}^1 * f_{k}^2|^r\right)^{1/r}\right\|_{L^p(\R)} \leq  \left\| \left(\sum_{k=1}^{n} |f_{k}^1|^{r_1}\right)^{1/r_1} \right\|_{L^{q_1}(\R)} \left\| \left(\sum_{k=1}^{n} |f_{k}^2|^{r_2}\right)^{1/r_2} \right\|_{L^{q_2}(\R)}
$$
for any choice of functions $\{f_{k}^i\}_{k=1}^{n} \subset L^{q_i}(\R)$, $1\leq i \leq 2$.

\subsection*{Acknowledgements} The authors wish to thank C. Muscalu for his valuable comments regarding this work.

\end{document}